\documentclass[arxiv,reqno,twoside,a4paper,12pt]{amsart}

\usepackage[utf8x]{inputenc}
\usepackage{xcolor}
\usepackage{tikz}
\usepackage{amsmath, verbatim}
\usepackage{amssymb,amsfonts,mathrsfs}
\usepackage[colorlinks=true,linkcolor=blue,citecolor=blue]{hyperref}
\usepackage[driver=pdftex,heightrounded=true,centering]{geometry}
\usepackage{longtable}
\usepackage{tabularx}
\usepackage{multirow}
\usepackage{caption}

\usepackage{times}
\usepackage[T1]{fontenc}
\usepackage{color}
\usepackage{array}
\usepackage{graphicx}
\usepackage{setspace}
\usepackage{stmaryrd}
\usepackage{esint}

\usepackage[scaled]{helvet} 
\usepackage{courier} 
\usepackage[mathbf]{euler}

\theoremstyle{plain}

\newtheorem{thm}{Theorem}[section]
\newtheorem{prop}[thm]{Proposition}
\newtheorem{cor}[thm]{Corollary}
\newtheorem{defn}[thm]{Definition}
\newtheorem{remark}[thm]{Remark}
\newtheorem{lem}[thm]{Lemma}

\theoremstyle{definition}

\theoremstyle{plain}

\newcommand{\R}{\mathbb{R}}

\newcommand{\C}{\mathbb{C}}

\newcommand{\N}{\mathbb{N}}

\newcommand{\dimn}{\mathrm{dim}}
\newcommand{\grad}{\mathrm{grad}}
\newcommand{\supp}{\mathrm{supp}}

\newcommand{\scal}{\mathrm{scal}}
\newcommand{\ric}{\mathrm{Ric}}

\newcommand{\dv}{\text{ }dV}

\DeclareMathOperator{\cH}{\mathscr{H}}
\DeclareMathOperator{\A}{\alpha}
\DeclareMathOperator{\w}{\omega}
\DeclareMathOperator{\V}{\mathcal{V}}
\DeclareMathOperator{\cC}{\mathscr{C}}
\DeclareMathOperator{\dom}{\mathscr{D}}
\DeclareMathOperator{\wx}{\widetilde{x}}
\DeclareMathOperator{\wt}{\widetilde{t}}
\DeclareMathOperator{\wz}{\widetilde{z}}

\DeclareMathOperator{\ff}{\textup{ff}}
\DeclareMathOperator{\tf}{\textup{tf}}
\DeclareMathOperator{\td}{\textup{td}}
\DeclareMathOperator{\rf}{\textup{rf}}
\DeclareMathOperator{\lf}{\textup{lf}}

\DeclareMathOperator{\ho}{\mathcal{C}^{\A}_{\textup{ie}}}
\DeclareMathOperator{\hho}{\mathcal{C}^{2+\A}_{\textup{ie}}}

\DeclareMathOperator{\hok}{\mathcal{C}^{k, \A}_{\textup{ie}}}

\numberwithin{equation}{section}

\definecolor{qqwuqq}{rgb}{0,0,0}

\begin{document}

\date{\today}

\title[Entropies for Manifolds with conical Singularities]
{Perelman's Entropies for Manifolds with conical Singularities}

\author{Klaus Kr\"oncke}
\address{University Hamburg, Germany} 
\email{klaus.kroencke@uni-hamburg.de} 

\author{Boris Vertman} 
\address{Universit\"at Oldenburg, Germany} 
\email{boris.vertman@uni-oldenburg.de}

\thanks{Partial support by DFG Priority Programme "Geometry at Infinity"}

\subjclass[2010]{Primary 53C44; Secondary 53C25; 58J05.}
\keywords{Perelman's entropy, Ricci flow, Ricci solitons,  singular spaces}

\begin{abstract}
In this paper we discuss Perelman's $\lambda$-functional, Perelman's Ricci shrinker entropy as well as the Ricci expander entropy
on a class of manifolds with isolated 
conical singularities. On such manifolds, a singular Ricci de Turck 
flow preserving the isolated conical singularities
exists by our previous work. We prove that the entropies are monotone along the singular 
Ricci de Turck flow. We employ these entropies 
to show that in the singular setting, Ricci solitons are gradient 
and that steady or expanding Ricci solitons are Einstein.
\end{abstract}

\maketitle

\tableofcontents

\section{Introduction and statement of the main results}

Perelman's utilization of Ricci flow in his proof of
Thurston's geometrization and the Poincare conjectures has dramatically increased
the interest in the research area of geometric flows. Perelman \cite{Perelman} 
introduced two important entropies, which are monotone along the Ricci flow and are
constant precisely on steady and shrinking solitions. More precisely, let $(M,g)$ 
be a compact Riemannian manifold of dimension $m$. We denote its scalar curvature by $\scal (g)$.
Perelman's $\lambda$-functional is then defined as follows. 
\begin{equation}\label{entropy-Perelman}
\lambda(g) :=\inf \left\{\int_M(\scal (g)+|\nabla f|^2_g)e^{-f}\dv_g \mid 
f\in C^\infty(M), \int_M e^{-f}\dv_g=1 \right\}.
\end{equation} 
Perelman's shrinker entropy is defined in three steps. First we
define the $\mathcal{W}_-$-functional $\mathcal{W}_-(g,f,\tau)$ for any $f\in C^\infty(M)$ and $\tau>0$
\begin{equation}\label{W-functional}
\mathcal{W}_-(g,f,\tau) :=\frac{1}{(4\pi\tau)^{m/2}}\int_M [\tau(|\nabla f|^2_g+\scal (g))+f-m]\, e^{-f}\dv_g.
\end{equation}
The $\mu_-$-functional is
\begin{equation}\label{shrinker-entropy}
\mu_-(g, \tau) :=\inf \left\{ \mathcal{W}_-(g,f,\tau) \mid  
f\in C^\infty(M), \frac{1}{(4\pi\tau)^{m/2}} \int_M e^{-f}\dv_g=1 \right\}.
\end{equation} 
Then the shrinker entropy is defined by 
\begin{equation}\label{entropy-Perelman-2}
\nu_-(g):= \inf \left\{ \nu_-(g,\tau) \mid \tau > 0\right\}
\end{equation}
and this latter infimum is finite and realized by a parameter $\tau_g$ if $\lambda(g)>0$, see \cite[Corollary 6.34]{CCG}.
These entropies are employed decisively in the analysis of 
Ricci solitions, which are Riemannian metrics $g$ such that for its 
Ricci curvature $\textup{Ric}(g)$, some vector field $X$, the Lie-derivative 
$\mathcal{L}_X$ and a positive constant $c>0$, the following equations are satisfied
\begin{equation}\label{def_soliton}
\begin{split}
& \textup{Ric}(g) + \mathcal{L}_X g = 0 \quad \textup{(steady Ricci soliton)}, \\
& \textup{Ric}(g) + \mathcal{L}_X g = c\, g \quad \textup{(shrinking Ricci soliton)}.
\end{split}
\end{equation}
Any steady Ricci soliton is up to a diffeomorphism a constant solution to the 
Ricci flow  
\begin{equation}\label{Ricci-steady}
\partial_tg(t) = - 2 \, \textup{Ric} \, (g(t)), \quad g(0)=g.
\end{equation}
Any shrinking Ricci soliton is up to a diffeomorphism a constant solution to the 
normalized Ricci flow
\begin{equation}\label{Ricci-shrinking}
\partial_t g(t) = - 2 \, \textup{Ric} \, (g(t)) + 2c\, g(t), 
\quad g(0)=g.
\end{equation}
Recall also that a Ricci soliton is called gradient if $X=\nabla f$ for some function $f:M\to\R$.
The following results are due to Perelman.
\begin{thm}[\cite{Perelman}] Let $(M,g)$ be a compact smooth Riemannian manifold. Then
	\begin{enumerate}
		\item The $\lambda$-functional is non-decreasing along the Ricci flow 
		\eqref{Ricci-steady} and is constant on steady Ricci solitons $g$. Moreover, any steady Ricci soliton is Ricci flat.
		\\
		\item The shrinker entropy is non-decreasing along the normalized Ricci flow
		\eqref{Ricci-shrinking} and is constant on shrinking Ricci solitons. Moreover, any 
		shrinking Ricci soliton is gradient.
	\end{enumerate}
\end{thm}
Additionally, in order to discuss the expanding case, Feldman, Ilmanen and Ni \cite{FIN}
introduced the Ricci expander entropy as follows. Consider the $\mu_+$-functional
\begin{equation}\label{expander-entropy}
\mu_+(g,\tau) :=\inf \left\{ \mathcal{W}_+(g,f,\tau) \mid \ 
f\in C^\infty(M), \frac{1}{(4\pi\tau)^{n/2}} \int_M e^{-f}\dv_g=1 \right\},
\end{equation} 
where the $\mathcal{W}_+$-functional $\mathcal{W}_+(g,f,\tau)$ is defined by 
\begin{equation}\label{Wplus-functional}
\mathcal{W}_+(g,f,\tau) :=\frac{1}{(4\pi\tau)^{n/2}}\int_M [\tau(|\nabla f|^2_g+\scal (g))-f+m] \, e^{-f}\dv_g.
\end{equation}
The Ricci expander entropy is then defined by 
\begin{equation}
\nu_+(g):= \sup \left\{ \mu_+(g,\tau) \mid \tau > 0\right\}.
\end{equation}
This infimum is finite and realized by a parameter $\tau_g$ if $\lambda(g)<0$, see \cite[p.\ 10]{FIN}.
The Ricci expander entropy can be used to discuss expanding 
Ricci solitions, which are Riemannian metrics $g$ such that for some vector field $X$ and a positve constant $c>0$,
\begin{equation}\label{def_soliton2}
\textup{Ric}(g) + \mathcal{L}_X g =-c\, g.
\end{equation} 
Any expanding Ricci soliton is up to a diffeomorphism a constant solution to the 
normalized Ricci flow
\begin{equation}\label{Ricci-expanding}
\partial_t g(t) = - 2 \, \textup{Ric} \,( g(t)) - 2c\, g(t), 
\quad g(0)=g.
\end{equation}
The following results are due to Feldman, Ilmanen and Ni.
\begin{thm}[\cite{FIN}] Let $(M,g)$ be a compact smooth Riemannian manifold. Then
the Ricci expander entropy  is non-decreasing along the normalized Ricci flow
\eqref{Ricci-expanding} and is constant on expanding Ricci solitons. Moreover, any 
expanding Ricci soliton is negative Einstein.
\end{thm}

Since the three entropies $\lambda, \nu_\pm$ are diffeomorphism invariant, instead of the 
Ricci flow one can equivalently consider the Ricci de Turck flow, which is given by the following equation.
\begin{equation}\label{RDF-0}
\partial_t g(t) = -2 \, \textup{Ric} (g(t)) + \mathcal{L}_{W(t)} g(t), \quad g(0) = g_0,
\end{equation}
where $W(t)$ is the de Turck vector field defined in terms of the Christoffel symbols for 
the metrics $g(t)$ and a reference metric $\widetilde{g}$\footnote{The reference metric 
$\widetilde{g}$ is often taken as the initial metric $\widetilde{g} = g_0$. In \cite{KrVe} the 
authors took the reference metric $\widetilde{g}$ as a Ricci flat conical metric, while the initial 
metric $g_0$ was a sufficiently small perturbation of $\widetilde{g}$.} 
\begin{equation}
W(t)^k = g(t)^{ij} \left(\Gamma^k_{ij}(g(t)) - \Gamma^k_{ij}(\widetilde{g})\right). 
\end{equation}
The de Turck vector field $W(t)$ yields a one parameter family of diffeomorphisms $\phi(t)$
and the pullback $\phi(t)^*g(t)$ solves the Ricci flow \eqref{Ricci-steady}. 
The normalized Ricci de Turck flow is defined, either by subtracting or by adding $2g(t)$
on the right hand side of the evolution equation in \eqref{RDF-0}, which yields after pullback by
the family of diffeomorphisms $\phi(t)$ a solution to \eqref{Ricci-shrinking} or \eqref{Ricci-expanding}, 
respectively. \medskip

The aim of this work is to extend these results to the setting of singular spaces.
Some preliminary results have been obtained by Dai and Wang \cite{DW1,DW2}, who 
established existence of minimizers for the Perelman's entropies  under certain conditions, 
but could not proceed further due to a priori weak regularity properties of these minimizers.

More precisely, they showed that the $\lambda$-functional and the shrinker entropy are defined if the scalar curvature of the Link satisfies $\scal (g_F)>n-1$. On the other hand, it was shown in recent work by Ozuch \cite{Ozu19} that both entropies are not defined (i.e. the infimum in the definition is $-\infty$) if $\scal (g_F)<n-1$. The case of equality is not understood in full generality.
 Building on top 
of the work of Dai and Wang, our first main result is an improved asymptotic expansion for the minimizers which is stated as follows:
\begin{thm}\label{asymptotics_minimizers2}
Let $(M^m,g)$ be a compact Riemannian manifold with an isolated conical singularity. Let $n=m-1$ and $(F^n,g_F)$ be the cross section of the conical part of the metric $g$ and assume that $\scal (g_F)=n(n-1)$.
	Let $\w_g$ be a minimizer in the definition of the $\lambda$-functional, shrinker or the expander entropy. 
	Then there exists an $\overline{\gamma}>0$ such that $\w_g$ admits a partial asymptotic expansion 
	$$\w_g(x,z) = \textup{const} + O(x^{{\overline{\gamma}}}), \ \textup{as} \ x\to 0,$$
	and moreover for $k\in\N$,
	$$|\nabla_g^k\w_g|_g(x,z)=O(x^{{\overline{\gamma}}-k}), \ \textup{as} \ x\to 0.$$
\end{thm}
This result is proved below in Corollary \ref{asymptotics_minimizers}, where the value of $\overline{\gamma}$ is made explicit. In contrast to the results in \cite{DW1,DW2} we are  able to compute the asympototics of all derivatives of the minimizing functions. In addition, we get much better asymptotic rates under the assumption on the scalar curvature. To complete the picture, we also analyze the expander entropy, which has not been considered by Dai and Wang in the conical case. These asymptotics have two important consequences which are stated in the following.
\begin{thm}\label{steadysolitonricciflat2}
Let $(M^m,g)$ be a compact Riemannian manifold with an isolated conical singularity. Let $n=m-1$ and $(F^n,g_F)$ be the cross section of the conical part of the metric $g$. Then the following statements hold. 
\begin{itemize}
\item[(i)] Suppose that $m\geq5$ and $\scal (g_F)=n(n-1)$. Then, if $(M,g)$ is a Ricci soliton, it is gradient. Moreover, if $(M,g)$ is steady or expanding, it is Einstein.
\item[(ii)] In dimension $m=4$, the assertions of part (i) hold if $\ric (g_F)=(n-1)g_F$ .
\end{itemize}
\end{thm}
This result is proved below in Theorem \ref{steadysolitonricciflat}, Theorem \ref{shrinkingsolitongradient2}
and Theorem \ref{expandingsolitoneinstein2}. In \cite[Corollary 9.2]{DW1}, it has been shown that steady gradient solitons are Ricci flat but no assertion about non-gradient solitons or the expanding case has been made. Our third main result studies the entropies along the singular Ricci de Turck flow.
\begin{thm}\label{monotonicity}Let $(M,g)$ be a compact Riemannian manifold with an isolated conical singularity of dimension $\dim M \geq 4$ and a tangentially stable cross section.
	\begin{itemize}
		\item[(i)] Then the $\lambda$-functional is nondecreasing along the Ricci de Turck flow preserving conical singularities and constant only along Ricci flat metrics.
		\item[(ii)] Whenever defined, the shrinker and the expanding entropies are nondecreasing along the (normalized) Ricci de Turck flow  preserving conical singularities and constant only along shrinking and expanding solitons, respectively.
	\end{itemize}
	\end{thm}
This statement is proved below in Theorem \ref{monotonicity_lambda}, Theorem \ref{monotonicity_nu_minus}
and Theorem \ref{monotonicity_nu_minus2}.
All the above results also extend to the case of finitely many isolated conical singularities in an obvious way.

The present paper can also be seen as a continuation of a research program on the Ricci flow 
preserving the given singularity structure, initiated in two dimensions by Mazzeo, Rubinstein and Sesum \cite{MRS}, Yin \cite{Yin:RFO}, and in general dimensions by the second author jointly with Bahuaud in \cite{BV}, \cite{BV2}.

Let us mention additional references of direct importance to the present discussion. In 
the preceding work \cite{Ver-Ricci}, the second named author establishes short time existence 
of Ricci (de Turck) flow preserving an edge singularity, under an analytic assumption 
of tangential stability for the cross-section of the cone. 
Our joint work \cite{KrVe} then established stability of the singular Ricci flow for small perturbations 
of metrics with isolated conical singularities and provided an explicit characterization of tangential stability. 
\medskip

\emph{Acknowledgements.} The authors thank the Priority programme "Geometry at Infinity" of
the German Research Foundation for financial and intellectual support.
The second author thanks Matthias Lesch for useful discussions.

\section{Preliminaries on manifolds with conical singularities}

In this section we gather the preliminary results obtained by the authors 
in \cite{Ver-Ricci} and \cite{KrVe}, as well as by Dai and Wang in \cite{DW1} and \cite{DW2}.

\subsection{Isolated conical singularities} \ \medskip

\begin{defn}\label{cone-metric}
Consider a compact smooth manifold $\overline{M}$ with boundary 
$\partial M = F$ and open interior denoted by $M$. 
Let $\overline{\cC(F)}$ be a tubular neighborhood of the boundary, 
with open interior $\cC(F) = (0,1)_x \times F$, where $x$ is a 
defining function of the boundary. Consider a smooth Riemannian metric
$g_F$ on the boundary $F$ with $n=\dim F$. An incomplete Riemannian metric $g$ on $M$ 
with an isolated conical singularity is then defined to be smooth away from the
boundary and 
\begin{equation*}
g \restriction \cC(F) = dx^2 + x^2 g_F + h,
\end{equation*}
where the higher order term $h$ is smooth on $\cC(F)$ with the following asymptotics 
at $x=0$. Let $\overline{g} = dx^2 + x^2 g_F$ denote the exact conical part of the metric $g$ over
$\cC(F)$ and $\nabla_{\overline{g}}$ the corresponding Levi Civita connection. Then we require that
for some $\gamma > 0$ and all integer $k \in \N_0$ the pointwise norm 
\begin{align}\label{higher-order}
| \, x^k \nabla_{\overline{g}}^k h \, |_{\overline{g}} = O(x^\gamma), \ x\to 0. 
\end{align}
\end{defn}

\begin{remark}
We emphasize here that we do not assume that the higher order term $h$
is smooth up to $x=0$ and do not restrict the order $\gamma>0$. In that sense
the notion of conical singularities in the present discussion is more 
general than the classical notion of conical singularities where $h$ is usually assumed
to be smooth up to $x=0$ with $\gamma = 1$.
\end{remark}

We call $(M,g)$ a compact space with an isolated conical singularity, 
or a \emph{conical manifold}\footnote{The reader should not confuse the notion of conical manifolds here, 
with the notion of \emph{cone-like} manifolds in \cite{DW1}. The cone-like manifolds in \cite{DW1} refer to manifolds 
with exact conical singularities, where $h\equiv 0$. The notion of conical manifolds in the present paper allows for 
perturbations $h$ as in \eqref{higher-order}, despite similarity in language with \cite{DW1}.}. 
The definition naturally extends to conical manifolds 
with finitely many isolated conical singularities. Since the analytic arguments are local 
in nature, we may assume without loss of generality that $M$ has a single conical singularity only.
\medskip

Our arguments on the Ricci flow rely on the assumption of tangential stability, which has been characterized
equivalently in our previous work \cite[Theorem 1.3]{KrVe}. We take this characterization as a definition.
 
\begin{defn}
	Let $(F,g_F)$ be a compact Einstein manifold of dimension $n\geq 2$ with constant $(n-1)$. 
	We write $\Delta_E$ for its Einstein operator, and denote the Laplace Beltrami 
	operator by $\Delta$. Then $(F,g_F)$ is said to be tangentially stable if 
	$\mathrm{Spec}(\Delta_E|_{TT})\geq0$ and $\mathrm{Spec}(\Delta)\setminus 
	\left\{0\right\}\cap (n,2(n+1))=\varnothing$. Similarly, $(F,g_F)$ is strictly tangentially 
	stable if and only if $\mathrm{Spec}(\Delta_E|_{TT})>0$ and $\mathrm{Spec}
	(\Delta)\setminus \left\{0\right\}\cap [n,2(n+1)]=\varnothing$. 
\end{defn}

Furthermore, \cite[Theorem 1.4]{KrVe} also provides an extensive list of examples 
where such an assumption of tangential stability is satisfied.
The assumption of tangential stability will be used here whenever our arguments 
employ the singular Ricci de Turck flow, which is constructed in 
\cite{Ver-Ricci} and \cite{KrVe} under this assumption. 

\subsection{Geometry of conical manifolds} \ \medskip

In this subsection we introduce the notions of b-tangent and b-cotangent bundles as well as 
some associated bundles, as in the b-calculus by Melrose \cite{Mel:TAP, Mel2}.
Consider local coordinates $(z)=(z_1,\ldots, z_n)$ on $F$, $\dim F = n$. Then $(x,z)$ defines 
local coordinates on the conical neighborhood $\cC(F) \subset M$. 
The Lie algebra of b-vector fields $\V_b$ consists by definition of smooth vector fields over $\overline{M}$ 
that are tangent to the boundary $\partial M = \{0\} \times F$. 
In local coordinates $(x,z)$, b-vector fields $\V_b$ are locally generated by 
\[
\left\{x\frac{\partial}{\partial x},  
\partial_z = \left( \frac{\partial}{\partial z_1},\dots, \frac{\partial}{\partial z_n} \right)\right\},
\]
with coefficients being smooth on $\overline{M}$. 
The b-tangent bundle ${}^bTM$ is defined by the requirement that the 
b-vector fields form a spanning set of section for ${}^bTM$, i.e. $\mathcal{V}_b=C^\infty(\overline{M},{}^bTM)$. 
Here, we are rather interested in the dual bundle, the b-cotangent bundle ${}^bT^*M$ 
which is generated locally by the following one-forms
\begin{align}\label{triv}
\left\{\frac{dx}{x}, dz_1,\dots,dz_n\right\}.
\end{align}
The differential form $\frac{dx}{x}$ is singular at $x=0$ in the usual sense, but is smooth as section of 
the b-cotangent bundle ${}^bT^*M$. Consider any extension of $x:\overline{\cC(F)} \to [0,1]$ 
to a smooth function on $\overline{M}$, nowhere-vanishing on $M$. Then the incomplete 
b-tangent space ${}^{ib}TM$ is defined by the space of sections $x C^\infty(\overline{M},{}^{ib}TM) := 
C^\infty(\overline{M},{}^{b}TM)$. The central bundle, used throughout the present discussion will be 
the dual incomplete b-cotangent bundle 
${}^{ib}T^*M$ which is related to its complete counterpart ${}^bT^*M$ by 
\begin{align}
C^\infty(\overline{M},{}^{ib}T^*M) = 
x C^\infty(\overline{M},{}^{b}T^*M), 
\end{align}
with the spanning sections given locally in the singular neighborhood $\cC(F)$ by 
\begin{align}\label{triv2}
\left\{dx, x dz_1,\dots, x dz_n\right\}.
\end{align}
With respect to the notation we just introduced, the exact part $\overline{g}$ of the conical 
metric $g$ in the Definition \ref{cone-metric} is a smooth section of 
the vector bundle of the symmetric $2$-tensors of the incomplete
b-cotangent bundle ${}^{ib}T^*M$, i.e. $\overline{g} \in C^\infty (\textup{Sym}^2({}^{ib}T^*M)
\restriction \overline{\cC(F)})$.

\subsection{Weighted Sobolev spaces on conical manifolds} \ \medskip

We continue in the setup of a conical manifold $(M,g)$. 
Let $\nabla_g$ denote the corresponding Levi Civita covariant derivative. 
Let the boundary defining function $x:\cC(F) \to (0,1)$ be extended smoothly to $\overline{M}$, 
nowhere vanishing on $M$. We consider the space 
$L^2(M)$ of square-integrable scalar functions with respect to the volume form of $g$.
We define for any $s\in \N$ and any $\delta \in \R$ the weighted Sobolev space $H^s_\delta (M)$ as 
the closure of compactly supported smooth functions $C^\infty_0(M)$ under
\begin{align}\label{Sobolev-norm}
\left\|u\right\|_{H^s_\delta} := \sum_{k=0}^s \| x^{k-\delta} 
  \nabla^k_g u \|_{L^2}.
\end{align}
Note that $L^2(M,E) = H^0_{0}(M,E)$ by construction. 

\begin{remark}
An equivalent norm on the weighted Sobolev space $H^s_\delta (M)$ can be defined 
for any choice of local bases $\{X_1, \ldots, X_m\}$ of $\V_b$ as follows. We omit the 
subscript $g$ from the notation of the Levi Civita covariant derivative and write 
\begin{align}
\left\|u\right\|_{H^s_\delta} = \sum_{k=0}^s \sum_{(j_1, \cdots, j_k)} \| \, x^{-\delta} 
\left( \nabla_{X_{j_1}} \circ \cdots \circ \nabla_{X_{j_k}} \right) u \, \|_{L^2}.
\end{align}
\end{remark}

The next lemma asserts that the 
weighted Sobolev space $H^s_\delta (M)$ is well-defined, i.e. independent of the choice of a conical metric $g$.

\begin{lem}\label{norm-equivalence}
Consider any two incomplete Riemannian metrics $g_0$ and $g$ with isolated conical 
singularities, as in Definition \ref{cone-metric}. Then the weighted Sobolev norms $\left\| \cdot \right\|_{0}$
and $\left\| \cdot \right\|$ defined in \eqref{Sobolev-norm} with respect to $g_0$ and $g$, respectively, 
are equivalent. 
\end{lem}

\begin{proof}
We first prove the statement for $g_0$ being an exact conic
metric $g_0\restriction \cC(F) = dx^2 + x^2 g_F$ near the conic singularity, and $g$ being its 
higher order perturbation. We begin with the observation that by \eqref{higher-order} for any integer $k\in \N_0$
we find as $x\to 0$
\begin{equation*}
\begin{split}
&| \, \nabla^k_{g_0} \left(g-g_0\right) \, |_{g_0} = O(x^{-k + \gamma}), \\
&| \, \nabla^k_{g_0} \left(g-g_0\right) \, |_{g} = O(x^{-k + \gamma}).
\end{split}
\end{equation*}
Consequently, writing $\Phi := \nabla_g - \nabla_{g_0}$, we obtain as $x\to 0$
$$
| \, \nabla_{g_0}^k \Phi \, |_{g_0} = O(x^{-k-1+\gamma}), \quad 
| \, \nabla_{g_0}^k \Phi \, |_{g} = O(x^{-k-1+\gamma}).
$$
Given any $k \in \N_0$, we want to express $\nabla_g$ in terms of 
$\nabla_{g_0}$. Consider any $\ell \in \N_0$ such that $k \geq \ell$. 
Let us write 
$$
J_{k, \ell} := \{ (m_1, \cdots, m_q) \subset \N_0 \, | \,  
\sum_{i=1}^q (m_i + 1) = k-\ell, q \in \N \}.
$$ 
Then we find by induction for any $k\in \N_0$
$$
\nabla_g^k = \sum_{\ell = 0}^k \sum_{J_{k, \ell}} 
\nabla_{g_0}^{m_1} \Phi \circ \cdots \circ \nabla_{g_0}^{m_1} \Phi \circ \nabla^\ell_{g_0}.
$$
We can now compute for any $u \in H^s_\delta(M)$ as $x\to 0$
\begin{align*}
\sum_{k=0}^s x^{k-\delta} | \, \nabla^k_g u \, |_g & = \sum_{k=0}^s x^{k-\delta} \sum_{\ell = 0}^k \sum_{J_{k, \ell}} 
| \, \nabla_{g_0}^{m_1} \Phi \circ \cdots \circ \nabla_{g_0}^{m_1} \Phi \circ \nabla^\ell_{g_0} u \, |_g
\\ & = \sum_{k=0}^s x^{k-\delta} \sum_{\ell = 0}^k \sum_{J_{k, \ell}} 
O(x^{\ell - k})| \, \nabla^\ell_{g_0} u \, |_g \\ & = \sum_{\ell=0}^s O(x^{\ell-\delta})
| \, \nabla^\ell_{g_0} u \, |_g = \sum_{\ell=0}^s O(x^{\ell-\delta})
| \, \nabla^\ell_{g_0} u \, |_{g_0}.
\end{align*}
We conclude that the weighted Sobolev norms defined with respect to $g$ and $g_0$
are equivalent. Since the computation above holds for arbitrary $\gamma \geq 0$, 
the statement holds for any two metrics in Definition \ref{cone-metric}.
\end{proof}

\subsection{Weighted H\"older spaces on conical manifolds}\label{spaces-section} \ \medskip

This section is basically a repetition of the corresponding definitions in \cite{Ver-Ricci}
in the case of isolated conical singularities, cf. \cite{KrVe}. We consider a manifold $(M,g)$ 
with isolated conical singularities. Due to the local structure of constructions, we may assume 
without loss of generality that we have just one conical singularity. All constructions extend easily to the case of 
multiple conical ends.\medskip

The H\"older spaces presented here, are the time-independent versions of the spaces
that arise in the construction of the singular Ricci de Turck flow 
in the preceding work \cite{Ver-Ricci} and \cite{KrVe}. Regularity in time is not relevant 
in the present discussion, and therefore we decided to simplify the setting by presenting
only the time-independent H\"older spaces here.

\begin{defn}\label{hoelder-A}
Consider the distance function $d_M(p,p')$ between any two points $p,p'\in M$, 
defined with respect to the conical metric $g$, and given in terms of the local coordinates 
$(x,z)$ over the singular neighborhood $\cC(F)$ equivalently by
\begin{align*}
d_M((x,z), (x',z'))=\left(|x-x'|^2+(x+x')^2|z-z'|^2\right)^{\frac{1}{2}}.
\end{align*}
The H\"older space $\ho(M), \A\in [0,1),$ consists of functions 
$u(p)$ that are continuous on $\overline{M}$ with finite $\A$-th H\"older 
norm
\begin{align}\label{norm-def}
\|u\|_{\A}:=\|u\|_{\infty} + \sup \left(\frac{|u(p)-u(p')|}{d_M(p,p')^{\A}}\right) <\infty.
\end{align}
The supremum is taken over all $(p,p') \in \overline{M}^2$\footnote{As explained in 
\cite{Ver-Ricci} we can assume without loss of generality that 
the tuples $(p,p')$ are always taken from within the same coordinate 
patch of a given atlas.}. 
\end{defn} 

We now extend the notion of H\"older spaces to sections of the  
vector bundle $S=\textup{Sym}^2({}^{ib}T^*M)$ of symmetric $2$-tensors. 

\begin{defn}\label{S-0-hoelder}
The Riemannian metric $g$ yields a fibrewise inner product on $S$, denoted again by $g$.
The H\"older space $\ho (M, S)$ consists by definition of all sections $\w$ of
$S$ which are continuous on $\overline{M}$, 
such that for any local orthonormal frame 
$\{s_j\}$ of $S$, the scalar functions $g(\w,s_j)$ are $\ho (M)$.
\medskip

The $\A$-th H\"older norm of $\w$ is defined using a partition of unity
$\{\phi_j\}_{j\in J}$ subordinate to a cover of local trivializations of $S$, with a 
local orthonormal frame $\{s_{jk}\}$ over $\supp (\phi_j)$ for each $j\in J$. We put
\begin{align}\label{partition-hoelder-2}
\|\w\|^{(\phi, s)}_{\A}:=\sum_{j\in J} \sum_{k} \| g(\phi_j \w,s_{jk}) \|_{\A}.
\end{align}
\end{defn}

Norms corresponding to different choices of $(\{\phi_j\}, \{s_{jk}\})$
are equivalent and we may drop the upper index $(\phi, s)$ from notation.
\medskip

We now turn to weighted and higher order H\"older spaces, where the weights 
are defined in terms of powers of the boundary
defining function $x:\cC(F) \to (0,1)$, extended smoothly to $\overline{M}$, 
nowhere vanishing on $M$. 

\begin{defn}\label{funny-spaces}
\begin{enumerate}
\item The weighted H\"older space for $\gamma \in \R$ is
\begin{align*}
&x^\gamma \ho(M, S) := \{ \, x^\gamma \w \mid \w \in \ho(M, S) \, \} 
\\ &\textup{with H\"older norm} \ \| x^\gamma \w \|_{\A, \gamma} := \|\w\|_{\A}.
\end{align*}
\item The hybrid weighted H\"older space for $\gamma \in \R$ is
\begin{align*}
&\ho(M, S)_{, \gamma}  := x^\gamma \ho(M, S)  \, \cap \, 
x^{\gamma + \A} \mathcal{C}^0_{\textup{ie}}(M, S) \\
&\textup{with H\"older norm} \  \| \w \|'_{\A, \gamma} := \|x^{-\gamma} 
\w\|_{\A} + \|x^{-\gamma-\A} \w\|_\infty.
\end{align*}
\item The weighted higher order H\"older spaces, which specify regularity of solutions 
under application of the Levi Civita covariant derivative $\nabla$ of $g$ on symmetric $2$-tensors
and time differentiation are defined for any 
$\gamma \in \R$ and $k \in \N$ by\footnote{Differentiation is a priori 
understood in the distributional sense.}
\begin{equation*}
\begin{split}
&\hok (M, S)_\gamma = \{\w\in \ho_{,\gamma} \mid  \nabla_{\V_b}^j
\w \in \ho_{,\gamma} \ \textup{for any} \ j \leq k \}, \\
&\hok (M, S)^b_\gamma = 
\{u\in \ho \mid  \nabla_{\V_b}^j u \in 
x^\gamma\ho \ \textup{for any} \ j \leq k\},
\end{split}
\end{equation*}
where the upper index b in the second space indicates the fact that despite
the weight $\gamma$, the solutions $u \in \hok (M, S)^b_\gamma$
are only bounded, i.e. $u\in \ho$. The corresponding H\"older norms are defined 
using local bases $\{X_i\}$ of $\V$ and $\mathscr{D}_k:=\{\nabla_{X_{i_1}} \circ \cdots \circ 
\nabla_{X_{i_j}} \mid j \leq k\}$ by
\begin{equation*}
\begin{split}
&\|\w\|_{k+\A, \gamma} = \sum_{j\in J} \sum_{X\in \mathscr{D}_k} \| X (\phi_j \w) \|'_{\A, \gamma}
+ \|\w\|'_{\A, \gamma}, \quad \textup{on} \ \hho (M, S)_\gamma, \\
&\|u\|_{k+\A, \gamma} = \sum_{j\in J} \sum_{X\in \mathscr{D}_k} \| X (\phi_j u) \|_{\A, \gamma}
+ \|u\|_{\A}, \quad \textup{on} \ \hho (M, S)^b_\gamma.
\end{split}
\end{equation*} 
\item In case of $\gamma=0$ we just omit the lower weight index and write
e.g. $\hok (M, S)$ and $\hok (M, S)^b$.
\end{enumerate}
\end{defn} 

The H\"older norms for different choices of local bases $\{X_1, \ldots, X_m\}$ of $\V_b$ and different choices
of Riemannian metrics $g$ with isolated conical singularities, are equivalent due to compactness of $M$ and $F$
by an argument as in Lemma \ref{norm-equivalence}. \medskip

The vector bundle $S$ decomposes into a direct sum of sub-bundles
\begin{align}
S= S_0 \oplus S_1, 
\end{align}
where the sub-bundle $S_0=\textup{Sym}_0^2({}^{ib}T^*M)$
is the space of trace-free (with respect to the fixed metric $g$) symmetric $2$-tensors,
and $S_1$ is the space of pure trace (with respect to the fixed metric $g$) symmetric 
$2$-tensors. The sub bundle $S_1$ is trivial real vector bundle over $M$ of rank 1.
\medskip

Definition \ref{funny-spaces} extends ad verbatim to sections of $S_0$ and $S_1$.
Since the sub-bundle $S_1$ is a trivial rank one real vector bundle, its sections
correspond to scalar functions. Since the regularity assumptions needed for sections of $S_0$
are slightly different from the regularity assumptions for sections of $S_1$, we introduce
hybrid H\"older spaces.

\begin{defn}\label{H-space}
Let $(M,g)$ be a compact conical manifold and 
assume that the conical cross section $(F,g_F)$ is
strictly tangentially stable. Then we define
\begin{equation*}
\cH^{k, \A}_{\gamma} (M, S) := \hok (M, S_0)_{\gamma}
\ \oplus \ \hok (M, S_1)^b_{\gamma}.
\end{equation*}
If $(F, g_F)$ is tangentially stable but not 
strictly tangentially stable, we set instead
\begin{equation*}
\cH^{k, \A}_{\gamma} (M, S) := \hok (M, S)^b_{\gamma}.
\end{equation*}
\end{defn}

\subsection{Singular Ricci de Turck flow} \ \medskip

We shall present here the short time existence result obtained by
the second author in \cite{Ver-Ricci} in a simple way, which is sufficient 
for the purpose of the present discussion. We consider a compact manifold $(M,g_0)$ with an isolated 
conical singularity. We study the Ricci de Turck flow with $g_0$ as the reference
and the initial metric 
\begin{equation}\label{RDT}
\partial_t g(t) = -2 \, \textup{Ric} (g(t)) + \mathcal{L}_{W(t)} g(t), \quad g(0) = g_0,
\end{equation}
where $W(t)$ is the de Turck vector field defined in terms of the Christoffel symbols for 
the metrics $g(t)$ and the reference metric $\widetilde{g}$ 
\begin{equation}
W(t)^k = g(t)^{ij} \left(\Gamma^k_{ij}(g(t)) - \Gamma^k_{ij}(\widetilde{g})\right). 
\end{equation}
While the reference metric $\widetilde{g}$ is usually taken as the initial metric $g_0$, 
in case of $\widetilde{g}$ being Ricci flat, the initial metric $g_0$ can be chosen as a sufficiently small perturbation 
of $\widetilde{g}$. This case has been the focal point of our work in \cite{KrVe}, where stability of the Ricci de
Turck flow in the singular setting has been addressed. \medskip

We need to impose additional conditions for the Ricci de Turck flow to exist in the singular setting. 
\begin{defn}\label{admissible}
Let $(M,g_0)$ be a conical manifold with an isolated conical singularity.
Then the conical metric $g_0$ is said to be \emph{admissible}, if it satisfies the following assumptions 
for $\gamma>0$ as in \eqref{higher-order}, all $k \in \N$  and some $\alpha \in (0,1)$. 
\begin{enumerate}
\item The cross section $(F,g_F)$ is assumed to be tangentially stable.
\item The Laplace Beltrami operator $\Delta_F$ of $(F, g_F)$ satisfies 
$\Delta_F \restriction (\ker \Delta_F)^{\perp}\geq \dim F$.
\item Let $\scal(g_0)$ denote the scalar curvature of $g$ and 
$\textup{Ric}^\circ (g_0)$ the trace-free part of the Ricci curvature tensor. Then we assume
\footnote{In view of Definition \ref{cone-metric} \eqref{higher-order} the condition (3) is satisfied
if the leading exact part $\overline{g} = dx^2+x^2g_F$ of the conical metric $g_0$ with 
$g_0\restriction \mathscr{C}(F)=\overline{g} + h$ is Ricci flat, and the higher order term $h$ 
not only satisfies \eqref{higher-order}, but in particular is an element of 
$\mathcal{C}^{k+3,\alpha}_{\textup{ie}}(M,S)_\gamma$.}
\begin{equation}
\begin{split}
&\scal(g_0) \in x^{-2+\gamma} \mathcal{C}^{k+1,\A}_{\textup{ie}}(M, S_1), \\
&\textup{Ric}^\circ (g_0) \in \mathcal{C}^{k+1,\A}_{\textup{ie}}(M, S_0)_{-2+\gamma}.
\end{split}
\end{equation}
\item For any $X_1, \ldots, X_4 \in C^\infty(\overline{M}, {}^{ib}TM)$
we have for the curvature $(0,4)$-tensor 
$$Rm(g_0)(X_1, X_2, X_3, X_4) \in x^{-2} \mathcal{C}^{k+1,\A}_{\textup{ie}}(M).$$ 
\end{enumerate}
We call $g_0 + h$ an \emph{admissible perturbation} if $h \in \cH^{k+2, \A}_{\gamma} (M, S)$\footnote{In case
$h \in \cH^{k+2, \A}_{\gamma'} (M, S)$ with $\gamma' < \gamma$, we can simply replace $\gamma$ by $\gamma'$.}.
\end{defn}
\begin{remark}
	Note that if $(F,g_F)$ is Einstein, $\mathrm{(2)}$ follows from $\mathrm{(1)}$ by the Obata-Lichnerowicz eigenvalue estimate \cite{Ob62}.
	\end{remark}

The admissibility condition $\Delta_F \restriction (\ker \Delta_F)^{\perp}\geq \dim F$ sharpens the condition that $(F,g_F)$
is tangentially stable and is due to analytic arguments in the sections below; it is actually 
not needed for the result on the existence of Ricci de Turck flow below. Similarly, for the singular Ricci 
de Turck flow to exist, the admissibility conditions need not to be satisfied for all $k\in \N$, which would imply
smoothness of $g$ in the open interior $M$, but only for some fixed integer. 
We still impose these more general conditions in order to gather all the
conditions on the conical metric $g$ under the umbrella \emph{admissible} metrics.
\medskip

The main result of \cite[Theorem 4.1]{Ver-Ricci}, see also \cite[Theorem 1.2]{KrVe}, is the following theorem. 
\begin{thm}\label{RF}
Let $(M,g_0)$ be a conical manifold with an admissible metric $g_0$. Let the reference
metric $\widetilde{g}$ be either equal to $g_0$ or an admissible conical Ricci flat metric, in which case 
$g_0$ is assumed to be a sufficiently small perturbation of $\widetilde{g}$ in $\cH^{k+2, \A}_{\gamma} (M, S)$.
\medskip

Then there exists some $T>0$, such that the Ricci de Turck flow \eqref{RDT} 
with reference metric $\widetilde{g}$, starting at $g_0$ admits a solution $g(t), t\in [0,T]$, which is 
an admissible perturbation of $g_0$, i.e. $g(t) \in \cH^{k+2, \A}_{\gamma'} (M, S)$ for each $t$, all $k\in \N$ and some
$\gamma' \in (0,\gamma)$ sufficiently small.
\end{thm}
We point out that \cite[Theorem 4.1]{Ver-Ricci} actually addresses regularity 
of $g(t)$ in $t\in [0,T]$ as well, but this aspect is irrelevant in the present discussion.

\subsection{Perelman's entropies in the singular setting}\label{DW-results} \ \medskip

Given an admissible perturbation $g$ of the conical metric $g_0$, the pointwise trace of $g$ 
with respect to $g_0$, denoted as $\textup{tr}_{g_0} g$ is by definition of admissibility an element of the 
H\"older space $\hok (M, S_1)^b_{\gamma}$. As explained in \cite[\S 5]{Ver-Ricci}, 
$\textup{tr}_{g_0} g$, restricts at $x=0$ to a constant function $(\textup{tr}_{g_0} g)(0) = u_0>0$
along $F$. Setting $\widetilde{x} := \sqrt{u_0} \cdot x$, the admissible perturbation $g= g_0 + h$
attains the form 
$$g = d\widetilde{x}^2 + \widetilde{x}^2 g_F + \widetilde{h},$$
where $|\widetilde{h}|_g = O(x^\gamma)$ as $x\to 0$. Note that the leading part of the 
admissible perturbation $g$ near the conical singularity differs from the leading part of the 
admissible metric $g_0$ only by scaling. In particular our notion of admissible perturbations 
is included in the notion of metrics with isolated conical singularities in 
Definition \ref{cone-metric} and Dai-Wang \cite[Definition 2.1]{DW1}
and hence their arguments apply to any admissible perturbation $g$. 

We shall review the results of \cite{DW1} and \cite{DW2} here, which hold for singular Riemannian 
metrics in Definition \ref{cone-metric}.

\begin{remark}
The results of \cite{DW1} and \cite{DW2} are stated under the additional assumption that $\gamma \geq 1$ for the 
weight $\gamma$ in \eqref{higher-order}. The authors put the restriction $\gamma \geq 1$ in order
to ensure that $H^s_\delta(M)$-Sobolev norms defined with respect to an exact conical metric and its 
higher order perturbation are equivalent. However, the authors miss that for weighted (!) Sobolev spaces, 
the norms are equivalent by Lemma \ref{norm-equivalence} without any additional assumption on $\gamma$. 
Therefore we may state the results of \cite{DW1} and \cite{DW2} for general $\gamma > 0$. 
\end{remark}

\subsubsection{$\lambda$-functional on conical manifolds} \  \medskip

The $\lambda$-functional in \eqref{entropy-Perelman} can be 
rewritten after a substitution $\omega := e^{-f/2}$ in the following equivalent form,
where in the singular setting we minimize over $\omega \in H^1_1(M)$ instead of $C^\infty(M)$
\begin{equation*}
\lambda(g) =\inf \left\{\int_M(\scal (g) \omega^2+4|\nabla \omega|^2_g) \dv_g \mid 
\omega \in H^1_1(M), \omega > 0, \int_M \omega^2 \dv_g=1 \right\}.
\end{equation*}
The following result is due to Dai and Wang \cite[Theorem 1.4]{DW1}

\begin{thm}\label{DW-1}
Let $(M,g_0)$ be a conical manifold of dimension $m$ with an admissible metric $g_0$ 
and an admissible perturbation $g$. Denote by $\Delta_g$ the Laplace Beltrami 
operator of $g$, with the positive sign convention. Then $\lambda(g)$ is finite with the 
minimizer $\w_g$ solving the equation 
\begin{align}
4\Delta_g \w_g+\scal (g) \w_g=\lambda(g) \w_g,
\end{align}
and satisfying the asymptotics $\w_g(x) = o\left(x^{-\frac{m-2}{2}}\right)$ as $x\to 0$
\footnote{In case of an exact conical singularity, \cite{DW1} also obtain a full asymptotics of 
$\w_g$.}.
\end{thm}

\subsubsection{The shrinker entropy on conical manifolds} \ \medskip

The functional $\mathcal{W}_-(g,f,\tau)$ in \eqref{W-functional} can be 
rewritten after a substitution $\omega := e^{-f/2}$ in the following equivalent form
\begin{equation*}
\mathcal{W}_-(g,\w,\tau) := \frac{1}{(4\pi\tau)^{m/2}}\int_M [\tau(\scal (g)\cdot \w^2 + 4|\nabla \w|^2_g)
-2\w^2 \ln \w - m \w^2] \dv_g.
\end{equation*}
The Ricci shrinker entropy in \eqref{shrinker-entropy}
is then redefined in the singular setting by minimizing over $\omega \in H^1_1(M)$ instead of $C^\infty(M)$
\begin{equation*}
\mu_-(g, \tau) =\inf \left\{ \mathcal{W}_-(g,\w,\tau) \mid 
\w \in H^1_1(M), \w > 0, \frac{1}{(4\pi\tau)^{m/2}} \int_M \w^2 \dv_g=1 \right\}.
\end{equation*}
The following result is due to Dai and Wang \cite[Theorem 1.4]{DW2}.

\begin{thm}\label{DW-2}
Let $(M,g_0)$ be a conical manifold of dimension $m$ with an admissible metric $g_0$ 
and an admissible perturbation $g$. Denote by $\Delta_g$ the Laplace Beltrami 
operator of $g$, with the positive sign convention. Then for any fixed $\tau > 0$, the Ricci 
shrinker entropy $\nu_-(g, \tau)$ is finite with the 
minimizer $\w_g$ solving the equation 
\begin{align}
\tau(-4\Delta_g \w_g-\scal (g) \w_g)+2\log(\w_g) \w_g+(m+\nu_-(g, \tau)) \w_g=0
\end{align}
and satisfying for any $\varepsilon > 0$ 
the asymptotics $\w_g(x) = o\left(x^{-\frac{m-2}{2} - \varepsilon}\right)$ as $x\to 0$.
\end{thm}
It is now shown exactly as in \cite[Corollary 6.34]{CCG}, that if $\lambda(g)>0$, the real number 
$\nu_-(g)=\inf \, \{ \mu_-(g,\tau) \mid \tau>0\}$ exists and is attained by a 
parameter $\tau_g$ and a minimizing function $\w_g$ satisfying the asymptotics from Theorem \ref{DW-2}.

\subsubsection{The expander entropy on conical manifolds} \ \medskip

The functional $\mathcal{W}_+(g,f,\tau)$ in \eqref{Wplus-functional} can be 
rewritten after a substitution $\omega := e^{-f/2}$ in the following equivalent form
\begin{equation*}
\mathcal{W}_+(g,\w,\tau) := \frac{1}{(4\pi\tau)^{m/2}}\int_M [\tau(\scal (g)\cdot \w^2 + \, 4|\nabla \w|^2_g)
+2\w^2 \ln \w + m \w^2] \dv_g.
\end{equation*}
The expander entropy in \eqref{expander-entropy}
is then redefined in the singular setting by minimizing over $\omega \in H^1_1(M)$ instead of $C^\infty(M)$
\begin{equation*}
\mu_+(g, \tau) =\inf \left\{ \mathcal{W}_+(g,\w,\tau) \mid 
\w \in H^1_1(M), \w > 0, \frac{1}{(4\pi\tau)^{m/2}} \int_M \w^2 \dv_g=1 \right\}.
\end{equation*}
The following result is an analogue of \cite[Theorem 1.4]{DW2}.
\begin{thm}\label{DW-3}Let $(M,g_0)$ be a conical manifold of dimension $m$ with an admissible metric $g_0$ 
and an admissible perturbation $g$. Denote by $\Delta_g$ the Laplace Beltrami 
operator of $g$, with the positive sign convention. Then for any fixed $\tau > 0$, the Ricci 
expander entropy $\nu_-(g, \tau)$ is finite with the 
minimizer $\w_g$ solving the equation 
\begin{align}\label{eulerlagrangeexpander}
\tau(-4\Delta_g \w_g-\scal (g) \w_g)-2\log(\w_g) \w_g+(-m+\nu_+(g, \tau)) \w_g=0
\end{align}
and satisfying for any $\varepsilon > 0$ 
the asymptotics $\w_g(x) = o\left(x^{-\frac{m-2}{2} - \varepsilon}\right)$ as $x\to 0$.
\end{thm}
\begin{proof}[Sketch of proof]
We now follow the strategy in \cite[p. 12-14]{DW2}.
At first, there exist positive constants $A,C_1,C_2 > 0$ such that
\begin{align*}
C_1\left\| \w \right\|_{H^1_1}^2\leq \int_M \left[4|\nabla \w|^2+(A+\scal(g)) \w^2\right]\dv_g\leq C_2\left\|\w\right\|_{H^1_1}^2.
\end{align*}
see \cite[(4.7)]{DW2}. By \cite[Lemma 4.1]{DW2} and since the function $x\mapsto x\log x$ has a lower bound, there exists a constant $C_3<0$ and for any $\epsilon>0$ a constant $C_4=C_4(\epsilon)>0$ such that
\begin{align*}
C_3\leq \int_M 2\w^2\log \w\dv_g\leq \epsilon \int_M |\nabla\w|^2\dv_g+C_4
\end{align*}
for all $\w\in H_1^1(M),\text{ }\w>0,\text{ }\int_M \w^2 \dv_g=1$.
Thus for $\tau>0$ fixed, 
\begin{align*}
-\infty<C_5+C_6\left\|\w\right\|_{H^1_1}^2\leq \mathcal{W}_+(g,\w, \tau)\leq C_7\left\|\w\right\|_{H^1_1}^2+C_8
\end{align*} 
for all $\w\in H_1^1(M),\text{ }\w>0,\text{ }\int_M \w^2 \dv_g=1$ and some constants $C_5<0$ and $C_6,C_7,C_8>0$. Let now $\w_i$ be a minimizing sequence for $\mu_+(g,\tau)$. Then $\left\|\w_i\right\|_{H^1_1}$ is uniformly bounded and by passing to a subsequence, we may assume the existence of $\w_0\in H^1_1(M)$ such that $\w_i$ converges to $\w_0$ weakly in $H_1^1$ and strongly in $L^2$. From that, we get in particular that $\w_0\geq0$ almost everywhere and that $\left\|\w_0\right\|_{L^2}=1$. To show that $\w_0$ is a minimizer in the definition of $\mu_+(g,\tau)$ it suffices to show that $\mathcal{W}_+(g,\w_i,\tau)\to \mathcal{W}_+(g,\w_0,\tau)$ which is shown exactly as in \cite[p. 13-14]{DW2}. This minimizer is a weak solution of the Euler-Lagrange equation \eqref{eulerlagrangeexpander}.
Smoothness is then shown by standard arguments, see e.g. \cite[p. 178]{AH11}. To show uniqueness, observe that by substituting $v=\w^2$, $\mathcal{W}_+$ can be rewritten as
\begin{align*}
\mathcal{W}_+(g,v,\tau)=\int_M \left[\tau(4|\nabla v^{1/2}|^2+\scal(g)\cdot v)+v\log v+mv\right]\dv
\end{align*}
which is strictly convex on the cone of smooth functions satisfying $\int_Mv\dv=1$ and $v>0$, see \cite[p. 9]{FIN}. The asymptotics of the minimizers is shown exactly as in \cite[Section 5]{DW2}.
\end{proof}
It is now shown exactly as in \cite[p.\ 10]{FIN}, that if $\lambda(g)<0$, $\nu_+(g)=
\sup \, \{\mu_+(g,\tau) \mid \tau>0\}$ exists and is attained by a 
parameter $\tau_g$ and a minimizing function $\w_g$ satisfying the 
asymptotics from Theorem \ref{DW-3}.

\section{Essential self-adjointness of the Laplace Beltrami operator} \medskip

Consider an incomplete Riemannian manifold $(M,g)$ with 
a conical metric $g$ in the sense of Definition \ref{cone-metric}.
In the singular neighborhood $\cC(F) = (0,1) \times F$ of the conical singularity such a metric $g$ 
is given by 
$$
g \restriction \cC(F) = dx^2 + x^2 g_F + h,
$$
where $(F,g_F)$ is a closed smooth Riemannian manifold and 
$h$ is a higher order term in the sense of \eqref{higher-order}.
In this section we shall establish essential self-adjointness of the
Laplace Beltrami operator $\Delta$, where notably $h$ is not assumed to 
be smooth and usual arguments do not apply. In fact we consider a slightly 
more general case of a Schr\"odinger operator for any $q\in \R$
$$
L := \Delta + q \cdot \scal (g).
$$
In case of $q\neq 0$ we additionally impose the admissibility assumption of Definition \ref{admissible}.
We define the maximal closed extension of $\Delta$ in $L^2(M)$ with domain
\begin{equation}
\dom(\Delta_{\max}) := \{\w \in L^2(M) \mid \Delta \w \in L^2(M)\},
\end{equation}
where $\Delta \w$ is defined distributionally.
We may also define the minimal closed extension of $\Delta$ in $L^2(M)$ as the domain of 
the graph closure of $\Delta$ acting on smooth compactly supported functions 
$C^{\infty}_0(M)$. More precisely, the minimal domain is defined by
\begin{equation*}
\begin{split}
\dom(\Delta_{\min}) := \{\w \in \dom(\Delta_{\max}) \mid \exists (\w_n)_{n\in \N} \subset C^{\infty}_0(M): 
u_n:= \w - \w_n\\
\|u_n\|_{\Delta} := \|\Delta u_n\|^2_{L^2} + \|u_n\|^2_{L^2} \to 0 \ \textup{as} \ n\to \infty \}
\end{split}
\end{equation*}
The Friedrichs self-adjoint extension of $\Delta$ in $L^2(M)$ is defined as the intersection of
$\dom(\Delta_{\max})$ with the domain of the graph closure of the square root of 
$\Delta$ acting on $C^{\infty}_0(M)$. More precisely the Friedrichs domain is given by 
\begin{equation}
\begin{split}
\dom(\Delta^{\mathscr{F}}) := \{\w \in \dom(\Delta_{\max}) \mid \exists (\w_n)_{n\in \N} \subset C^{\infty}_0(M): 
u_n:= \w - \w_n\\
\|u_n\|_{\mathscr{F}} := (\Delta u_n, u_n)_{L^2} + \|u_n\|^2_{L^2} \to 0 \ \textup{as} \ n\to \infty\}.
\end{split}
\end{equation}
Consider an exact conical metric $g_0$, smooth in $M$ and 
given over $\cC(F)$ by
$$
g_0 \restriction \cC(F) = dx^2 + x^2 g_F.
$$
We shall write $\Delta_0$ for the Laplace Beltrami operator
of $g_0$. The maximal, minimal and Friedrichs domains for $\Delta_0$ are defined 
analogously and in fact are equal by the following classical result.

\begin{prop}\label{minimal-domain-1}  Let $\dim F = n \geq 3$. Then $\Delta_0$ is 
essentially self-adjoint and moreover
\begin{equation}
\begin{split}
\dom (\Delta_{0, \min}) = \dom (\Delta_{0, \max}) = \dom (\Delta^{\mathscr{F}}_{0}) = H^2_2(M). 
\end{split}
\end{equation}
\end{prop}

\begin{proof}
Let $(\lambda, \w_\lambda)_{\lambda}$ be the set of eigenvalues and 
corresponding eigenfunctions of the Laplace Beltrami operator $\Delta_F$ of $(F,g_F)$.
The Laplace Beltrami operator is non-negative, $\lambda \geq 0$, and we may define
\begin{equation}\label{nu}
\nu(\lambda) := \sqrt{\lambda + 
\left(\frac{n-1}{2}\right)^2}.
\end{equation}
Standard arguments, see e.g. \cite[Lemma 2.2]{MazVer} or \cite{KLP:FDG}, 
cf. the exposition in \cite{Ver:ZDR}, show that for each $\w \in \dom(\Delta_{0,\max})$
there exist constants $c^\pm_\lambda(\w), \nu(\lambda) \in [0,1)$, 
depending only on $\w$, such that $\w$ admits a partial asymptotic expansion as $x\to 0$
\begin{equation}\label{cone-asymptotics}
\begin{split}
\w & = \sum_{\nu(\lambda) = 0} \left(c^+_{\lambda}(\w) x^{ - \frac{(n-1)}{2}}
+ c^-_{\lambda}(\w) x^{ - \frac{(n-1)}{2}} \log(x) \right) \cdot \omega_\lambda 
\\ & + \sum_{\nu(\lambda) \in (0,1)} \left(c^+_{\lambda}(\w) x^{\nu(\lambda) - \frac{(n-1)}{2}}
+ c^-_{\lambda}(\w) x^{-\nu(\lambda) - \frac{(n-1)}{2}} \right) \cdot \omega_\lambda 
 + \widetilde{\w},
\end{split}
\end{equation}
where $\widetilde{\w} \in \dom(\Delta_{0, \min})$.
If $n\geq 3$, $\nu(\lambda)\geq 1$ for any eigenvalue $\lambda\geq 0$ of $\Delta_F$.
Hence minimal and maximal domains,
and consequently any self-adjoint domain including the Friedrichs extension, coincide.
\medskip

The statement now follows using 
elements of Melrose's $b$-calculus \cite{Mel:TAP}. The operator $x^2\Delta_0$ is an 
elliptic differential $b$-operator in the sense of Melrose \cite{Mel:TAP} and a central 
consequence of Melrose's $b$-calculus is existence of a parametrix $Q$, such that
$Q \circ x^2\Delta_0 = \textup{Id} - R$, where for any $\delta \in \R$ and $\ell \in \N_0$
\begin{equation}\label{QR}
\begin{split}
&Q: H^\ell_\delta (M) \to H^{\ell +2}_\delta (M), \\
&R: H^\ell_\delta (M) \to H^{\infty}_\delta (M) \cap \mathcal{A}_{\textup{phg}}(M). 
\end{split}
\end{equation}
Here $\mathcal{A}_{\textup{phg}}(M)$ denotes the space of smooth functions on $M$ that
admit a full asymptotic expansion as $x\to 0$ in terms of powers of $x$ and $\log (x)$
with smooth coefficients. Consider any $\w \in \dom (\Delta_{0,\max})$ with $\Delta_0 \w = f \in L^2(M)$.
Then, applying the parametrix $Q$ to both sides of the equation $x^2 \Delta_0 \w = x^2 f$
we obtain 
\begin{align*}
\w = R \w + Q (x^2 f) =: \w_1 + \w_3.
\end{align*}
By \eqref{QR} we find
\begin{align*}
&\w_1 \in H^{\infty}_0 (M) \cap \mathcal{A}_{\textup{phg}}(M), 
\qquad \w_2 \in H^{2}_{2} (M), \\
&\Delta_0 \w_1 \in H^{\infty}_{-2} (M) \cap \mathcal{A}_{\textup{phg}}(M), \qquad 
\Delta_0 \w_2 \in L^2(M).
\end{align*}
Since $\w_1  \in L^2(M) \cap \mathcal{A}_{\textup{phg}}(M)$ there exists $\mu > - \frac{n+1}{2}$ and $p \in \N_0$
such that 
\begin{equation*}
\begin{split}
\w_1(x,z) & \sim \A(z) \cdot x^\mu \log^p (x), \ \textup{as} \ x \to 0,
\end{split} 
\end{equation*}
for some $\A \in C^\infty(F)$. Applying $\Delta_0$ to the expansion above we obtain as $x\to 0$
\begin{equation}\label{Delta-w-1}
\begin{split}
\Delta_0 \w_1 (x,z) \sim \Bigl( (-\mu (\mu - 1) - n \mu) \A(z) + \Delta_F \A(z) \Bigr) \cdot x^{-2+\mu} \log^p (x).
\end{split} 
\end{equation}
Since $\Delta_0 \w, \Delta_0 \w_2 \in L^2(M)$, we conclude that $\Delta_0 \w_1 \in L^2(M)$. 
We now want to prove that $\w_1 \in 
H^2_\rho(M)$. For $n\geq 3$, the leading coefficient in \eqref{Delta-w-1} can vanish only 
if $\mu \geq 0$\footnote{The other case $\mu \leq -n+1$ is excluded due to $\mu > - \frac{n+1}{2}$.}, 
in which case $\w_1 \in H^2_2(M)$ as desired. Otherwise, $\Delta_0 \w_1 \in L^2(M)$
together with the expansion \eqref{Delta-w-1} implies $\mu > -\frac{n+1}{2} + 2$, 
in which case $\w_1 \in H^2_2(M)$ as well. We conclude $\w \in H^2_2(M)$.
This proves $\dom (\Delta_{0,\max}) \subseteq H^2_2(M)$. The converse inclusion is 
clear by definition and hence the statement follows.
\end{proof}

\begin{cor}\label{minimal-domain-2}  Let $\dim F = n \geq 3$. Then $L$ is 
essentially self-adjoint and moreover
\begin{equation}
\begin{split}
\dom (L_{\min}) = \dom (L_{\max}) = H^2_2(M). 
\end{split}
\end{equation}
\end{cor}

\begin{proof}
Let us now consider the difference 
\begin{equation}\label{V}
V:= L - \Delta_0 = \Delta - \Delta_0 + q \cdot \scal(g): 
H^k_{\rho} (M) \to H^{k-2}_{-2 + \rho + \gamma}(M).
\end{equation}
which is a bounded operator for any $k\in \N$ and any $\rho \in \R$, 
by the admissibility assumption in Definition \ref{admissible}.
Consequently for any $\varepsilon > 0$ there exists some $\delta > 0$
sufficiently small such that for any test function $\phi \in C^\infty_0(M)$ with 
compact support $\supp \, \phi \subset (0,\delta) \times F \subset M$ and any
$u \in C^\infty_0(M)$
\begin{equation}\label{V-estimate}
\| \phi V u\|_{L^2} \leq \sup_{p \in (0,\delta) \times F} x^\gamma(p) \cdot \| u \|_{H^2_2(M)}
\leq \varepsilon \cdot  \| u \|_{H^2_2(M)}.
\end{equation}
Fix $\varepsilon < 1$ and fix a corresponding $\delta > 0$. Proposition \ref{minimal-domain-1}
applies to $\Delta_0 + (1-\phi) V$ instead of $\Delta_0$, since $(1-\phi) V$ is a smooth differential
operator supported away from the conical singularity. In particular
$$
\dom_{\min}(\Delta_0 + (1-\phi) V) = \dom_{\max}(\Delta_0 + (1-\phi) V) = H^2_2(M).
$$
We now employ a result by Kato, \cite[Ch. V, \S 4, Theorem 4.5, p.289]{Kato}
which states the following: Let $T, S$ be two symmetric operators with same domain 
$\dom$ in a Hilbert space. Assume that for any $u \in \dom$
we have the inequality $\| (S-T)u \| \leq a \|u \| + b(\|Tu \| + \|S u\|)$ for $a \geq 0$ and $b\in [0,1)$. Then $S$ is 
essentially self-adjoint if and only if $T$ is. The domains of their unique self-adjoint extensions coincide.

Hence for $\varepsilon>0$ sufficiently small, we conclude
that $L = \Delta_0 + (1-\phi) V + \phi V$
is essentially self-adjoint with same self-adjoint domain $H^2_2(M)$.
\end{proof}

\section{Construction of a heat parametrix} \ \medskip

We consider the Laplace Beltrami operator $\Delta$ on a 
Riemannian manifold $(M,g)$ with isolated conical singularities in the 
sense of Definition \ref{cone-metric}. We consider for any $q\in \R$ the 
Schr\"odinger operator $L= \Delta + q \cdot \scal(g)$. In case of $q\neq 0$
we also impose the admissibility assumption in Definition \ref{admissible}.
In this subsection we consider the heat equation 
\begin{equation*}
(\partial_t + L) \, \w(t,p)  = 0, \ \w(0,p)= \w_0(p),
\end{equation*}
and explain how to construct its fundamental solution, 
following the heat kernel construction of \cite{MazVer}, which
is actually due to Mooers \cite{Moo} in the present conical setting. 
The solution is an integral convolution operator acting on compactly supported sections $\w$ by 
\begin{equation} \label{eqn:hk-on-functions}
\left(H \w\right) (t,p) = \int_M \left(H \left( t, p,\widetilde{p} \right),  
\w(\widetilde{p})\right)_g \dv (\widetilde{p}),
\end{equation}
The integral kernel of $H$ is a distribution on $M^2_h=\R^+\times \widetilde{M}^2$.
Consider the local coordinates near the corner in $M^2_h$ given by $(t, (x,z), (\widetilde{x}, \widetilde{z}))$, 
where $(x,z)$ and $(\widetilde{x}, \widetilde{z})$ are two copies of coordinates on $M$ near the conical singularity. 
The kernel $H(t, (x,z), (\wx,\wz))$ has non-uniform behaviour at the submanifolds
\begin{align*}
&A =\{ (t, (x, z), (\wx,\wz))\in M^2_h \mid t=0, \, x=\wx=0\}, \\
&D =\{ (t, p, \widetilde{p})\in M^2_h \mid t=0, \, p=\widetilde{p}\},
\end{align*}
which requires an appropriate blowup of the heat space $M^2_h$, 
such that the corresponding heat kernel lifts to a polyhomogeneous distribution 
in the sense of the following definition, which we cite from \cite{Mel:TAP} and \cite{MazVer}.

\begin{defn}\label{phg}
Let $\mathfrak{W}$ be a manifold with corners and $\{(H_i,\rho_i)\}_{i=1}^N$ an enumeration 
of its (embedded) boundaries with the corresponding defining functions. For any multi-index $b= (b_1,
\ldots, b_N)\in \C^N$ we write $\rho^b = \rho_1^{b_1} \ldots \rho_N^{b_N}$.  Denote by 
$\mathcal{V}_b(\mathfrak{W})$ the space of smooth vector fields on $\mathfrak{W}$ which lie
tangent to all boundary faces. A distribution $\w$ on $\mathfrak{W}$ is said to be conormal,
if $\w$ is a restriction of a distribution across the boundary faces of $\mathfrak{W}$, 
$\w\in \rho^b L^\infty(\mathfrak{W})$ for some $b\in \C^N$ and $V_1 \ldots V_\ell \w \in \rho^b L^\infty(\mathfrak{W})$
for all $V_j \in \mathcal{V}_b(\mathfrak{W})$ and for every $\ell \geq 0$. An index set 
$E_i = \{(\gamma,p)\} \subset {\mathbb C} \times {\mathbb N_0}$ 
satisfies the following hypotheses:

\begin{enumerate}
\item $\textup{Re}(\gamma)$ accumulates only at $+\infty$,
\item for each $\gamma$ there exists $P_{\gamma}\in \N_0$, such 
that $(\gamma,p)\in E_i$ for all $p \leq P_\gamma$,
\item if $(\gamma,p) \in E_i$, then $(\gamma+j,p') \in E_i$ for all $j \in {\mathbb N_0}$ and $0 \leq p' \leq p$. 
\end{enumerate}
An index family $E = (E_1, \ldots, E_N)$ is an $N$-tuple of index sets. 
Finally, we say that a conormal distribution $\w$ is polyhomogeneous on $\mathfrak{W}$ 
with index family $E$, we write $\w\in \mathscr{A}_{\textup{phg}}^E(\mathfrak{W})$, 
if $\w$ is conormal and if in addition, near each $H_i$, 
\[
\w \sim \sum_{(\gamma,p) \in E_i} a_{\gamma,p} \rho_i^{\gamma} (\log \rho_i)^p, \ 
\textup{as} \ \rho_i\to 0,
\]
with coefficients $a_{\gamma,p}$ conormal on $H_i$, polyhomogeneous with index $E_j$
at any intersection $H_i\cap H_j$ of hypersurfaces. 
\end{defn}

We review briefly the sequence of parabolic blowups as outlined in \cite{MazVer}. 
First we define the parabolic blowup $[M^2_h, A]$ 
(parabolic in the sense that we treat $\sqrt{t}$ as a smooth variable)
as the disjoint union of $M^2_h\backslash A$ with the interior spherical normal bundle of $A$ in $M^2_h$
under appropriate identifications, cf.  \cite{Mel:TAP}. 
The blowup $[M^2_h, A]$ is equipped with the minimal differential structure 
containing smooth functions in the interior of $M^2_h$ and polar coordinates 
on $M^2_h$ around $A$. The interior spherical normal bundle of $A$ defines a new boundary 
hypersurface $-$ the front face ff in addition to the previous boundary faces 
$\{x=0\}, \{\wx=0\}$ and $\{t=0\}$, which lift to rf (the right face), lf (the left face) and 
tf (the temporal face), respectively.  \medskip

\begin{figure}[h]
\begin{center}
\begin{tikzpicture}
\draw (0,0.7) -- (0,2);
\draw[dotted] (-0.1,0.7) -- (-0.1, 2.2);
\node at (-0.4,2) {t};

\draw(-0.7,-0.5) -- (-2,-1);
\draw[dotted] (-0.69,-0.38) -- (-2.05, -0.9);
\node at (-2.05, -0.6) {$x$};

\draw (0.7,-0.5) -- (2,-1);
\draw[dotted] (0.69,-0.38) -- (2.05, -0.9);
\node at (2.05, -0.6) {$\wx$};

\draw (0,0.7) .. controls (-0.5,0.6) and (-0.7,0) .. (-0.7,-0.5);
\draw (0,0.7) .. controls (0.5,0.6) and (0.7,0) .. (0.7,-0.5);
\draw (-0.7,-0.5) .. controls (-0.5,-0.6) and (-0.4,-0.7) .. (-0.3,-0.7);
\draw (0.7,-0.5) .. controls (0.5,-0.6) and (0.4,-0.7) .. (0.3,-0.7);

\draw (-0.3,-0.7) .. controls (-0.3,-0.3) and (0.3,-0.3) .. (0.3,-0.7);
\draw (-0.3,-1.4) .. controls (-0.3,-1) and (0.3,-1) .. (0.3,-1.4);

\draw (0.3,-0.7) -- (0.3,-1.4);
\draw (-0.3,-0.7) -- (-0.3,-1.4);

\node at (1.2,0.7) {\large{rf}};
\node at (-1.2,0.7) {\large{lf}};
\node at (1.1, -1.2) {\large{tf}};
\node at (-1.1, -1.2) {\large{tf}};
\node at (0, -1.7) {\large{td}};
\node at (0,0.1) {\large{ff}};
\end{tikzpicture}
\end{center}
\label{heat-incomplete}
\caption{The heat-space $\mathscr{M}^2_h$.}
\end{figure}

The heat-space $\mathscr{M}^2_h$ is obtained by a second parabolic blowup of  
$[M^2_h, A]$ along the diagonal $D$, lifted to a submanifold of $[M^2_h, A]$. 
We proceed as before by cutting out the lift of $D$ and replacing it with its spherical 
normal bundle, under appropriate identifications, which introduces a new boundary 
face $-$ the temporal diagonal td. The heat space $\mathscr{M}^2_h$ is equipped 
with the blowdown map $\beta: \mathscr{M}^2_h \to M^2_h$ and is illustrated in Figure 1. \medskip

Instead of polar coordinates on the heat space $\mathscr{M}^2_h$, we may consider 
a convenient replacement by projective coordinates near each 
boundary face. Their advantage is that local computations
are much easier in the projective coordinates, their disadvantage is that projective coordinates are 
not globally defined over the entire boundary face. Near the top corner of the front face ff, 
projective coordinates are given by
\begin{align}\label{top-coord}
\rho=\sqrt{t}, \  \xi=\frac{x}{\rho}, \ \widetilde{\xi}=\frac{\wx}{\rho}, \ z, \ \wz.
\end{align}
With respect to these coordinates, $\rho, \xi, \widetilde{\xi}$ are in fact the defining 
functions of the boundary faces ff, rf and lf respectively. 
For the bottom right corner of the front face, projective coordinates are given by
\begin{align}\label{right-coord}
\tau=\frac{t}{\wx^2}, \ s=\frac{x}{\wx}, \ z, \ \wx, \ \widetilde{z},
\end{align}
where in these coordinates $\tau, s, \widetilde{x}$ are
the defining functions of tf, rf and ff respectively. 
For the bottom left corner of the front face,
projective coordinates are obtained by interchanging 
the roles of $x$ and $\widetilde{x}$. Projective coordinates 
on $\mathscr{M}^2_h$ near temporal diagonal are given by 
\begin{align}\label{d-coord}
\eta=\frac{\sqrt{t}}{x}, \ S =\frac{(x-\wx)}{\sqrt{t}},
\ Z =\frac{\wx (z-\wz)}{\sqrt{t}}, \  x, \ z.
\end{align}
In these coordinates, tf is defined as the limit $|(S, Z)|\to \infty$, 
ff and td are defined by $\widetilde{x}, \eta$, respectively. 
The blow-down map $\beta: \mathscr{M}^2_h\to M^2_h$ is in 
local coordinates simply the coordinate change back to 
$(t, (x,z), (\widetilde{x},\widetilde{z}))$. \medskip

We now proceed with the definition of a heat calculus on $\mathscr{M}^2_h$,
which up to rescaling corresponds to the definition provided in \cite[Definition 3.1]{MazVer}.

\begin{defn}\label{heat-calculus-phg}
We define $\Psi^{\, \ell, p, E_{\textup{lf}}, E_{\textup{rf}}}_{\textup{phg}}(M)$
to be the space of all integral operators $A$ with Schwartz kernels $K_A$
that lift to a polyhomogeneous function $\beta^*K_A$ on $\mathscr{M}^2_h$ with 
\begin{enumerate}
\item the index set $(-\dim M - 2 + \ell + \N_0 \, ; \, 0)$ at the front face ff, 
\item the index set $(-\dim M + p + \N_0 \, ; \, 0)$ at the temporal diagonal td, 
\item the index sets $E_{\textup{lf}}, E_{\textup{rf}}$ at the left and right faces, respectively,
\item and vanishing to infinite order at the temporal face tf. 
\end{enumerate}
\end{defn}

This indeed defines a calculus in view of the following composition result, cf. 
\cite[Theorem 5.3]{MazVer} in the more general case of non-isolated conical singularities.
There the authors study the kernels under an additional unitary rescaling \cite[(2.1)]{MazVer},
leading to a shift of the index sets.

\begin{thm}\label{heat-calculus-phg-thm}
For index sets $E_{\textup{rf}}$ and $E'_{\textup{rf}}$ such that 
$E_{\textup{lf}}+E'_{\textup{rf}} > - 1 - n$, we have
$$\Psi^{\, \ell,p,E_{\textup{lf}}, E_{\textup{rf}}}_{\textup{phg}}(M) \circ 
\Psi^{\, \ell',\infty,E'_{\textup{lf}}, E'_{\textup{rf}}}_{\textup{phg}}(M) \subset 
\Psi^{\, \ell+\ell',\infty,P_{\textup{lf}}, P_{\textup{rf}}}_{\textup{phg}}(M),$$
where the front face expansion does not contain logarithmic terms and
\footnote{We define for any index sets $E_1, E_2$ the extended union by 
$E_1 \overline{\cup} E_2 := E_1 \cup E_2 \cup \{(z,p_1 + p_2 +1) 
\mid (z,p_1) \in E_1, (z,p_2) \in E_2\}$.} 
\begin{equation*}
\begin{split}
P_{\textup{lf}}&=E'_{\textup{lf}}\overline{\cup} (E_{\textup{lf}}+ \ell'), \\
P_{\textup{rf}}&=E_{\textup{rf}}\overline{\cup} (E'_{\textup{rf}}+ \ell).
\end{split}
\end{equation*} 
\end{thm}

Ultimately, we will have to deal with operators whose integral kernels are not necessarily
polyhomogeneous any longer, but are only 
conormal with certain bounds. We therefore we extend the definition of a 
heat calculus above to include the cases of interest as follows.

\begin{defn}\label{heat-calculus}
We define $\Psi^{\, \ell, p, \mu_1, \mu_2}(M)$
to be the space of all integral operators $A$ with Schwartz kernels $K_A$
that lift to conormal functions $\beta^*K_A$ on $\mathscr{M}^2_h$, vanishing to 
infinite order at tf and, writing $\rho_*$ for a defining function of the boundary face $*$
$$
\beta^*K_A = \rho_{\textup{ff}}^{- \dim M - 2 + \ell} 
\, \rho_{\textup{td}}^{- \dim M + p}
\, \rho_{\textup{lf}}^{\mu_1}
\, \rho_{\textup{rf}}^{\mu_2} 
 \, L^\infty(\mathscr{M}^2_h).
$$
If $E_{\textup{lf}} \geq \mu_1, E_{\textup{rf}} \geq \mu_2$, then the polyhomogeneous and 
conormal calculi are related by
$$
\Psi^{\, \ell, p, E_{\textup{lf}}, E_{\textup{rf}}}_{\textup{phg}}(M) 
\subseteq \Psi^{\, \ell, p, \mu_1, \mu_2}(M).
$$
\end{defn}

The composition result in Theorem \ref{heat-calculus-phg-thm} extends to the 
case of conormal integral kernels with bounds, compare e.g. \cite{AlGe},
using a version of Melrose's Pushforward theorem with bounds. We concludes with 
the following composition result. 

\begin{thm}\label{heat-calculus-thm}
For $\mu_1+\mu'_2 > -1-n$, we have
$$\Psi^{\, \ell, p, \mu_1, \mu_2}(M) \circ 
\Psi^{\, \ell', \infty, \mu'_1, \mu'_2}(M) \subset 
\Psi^{\, \ell+\ell',\infty,\mu''_1, \mu''_2}(M),$$
where we have set $\mu''_k := \min \{\mu_k, \mu'_k + \ell' \}$
for $k=1,2$.
\end{thm}

We may now proceed with constructing the fundamental solution to the heat equation.
Exactly as in the heat kernel construction in \cite[\S 3.2]{MazVer} we construct 
an initial heat parametrix as follows. We choose any smooth cutoff function $\chi \in C^\infty[0,\infty)$
such that $\chi \restriction [0,\varepsilon] \equiv 1$ and $\chi \restriction [2\varepsilon, \infty) \equiv 0$
for $\varepsilon > 0$ sufficiently small. Then we set in terms of projective coordinates \eqref{top-coord}
\begin{equation*}
\beta^* H_0 (\tau, \xi, z, \widetilde{\xi},\widetilde{z}) := \chi(\tau) \cdot 
\tau^{-n}(\xi \widetilde{\xi})^{\frac{n}{2}}
\bigoplus_\lambda \frac{(\xi \widetilde{\xi})^{\frac12}} {2\tau} I_{\nu(\lambda)}\left(\frac{\xi\widetilde{\xi}}{2\tau}\right)
e^{-\frac{\xi^2+\widetilde{\xi}^2}{4\tau}}\phi_{\lambda}(z)\phi_{\lambda}(\widetilde{z}),
\end{equation*}
where $\lambda$ runs over the spectrum of $\Delta_F$, counted with multiplicities, with corresponding
eigenfunctions given by $\phi_{\lambda}$, $\nu(\lambda)$ is defined by \eqref{nu} and 
$I_{\nu(\lambda)}$ denotes the modified Bessel function of first kind. The factor 
$\tau^{-n}(\xi \widetilde{\xi})^{\frac{n}{2}}$ in fact does not appear in \cite[(3.9)]{MazVer}, 
since the analysis there is performed under a unitary rescaling, cf. \cite[(2.6)]{BV}.
We can regard $\beta^* H_0$ as a polyhomogeneous function on $\mathscr{M}^2_h$, where identifying
the integral kernel $H_0$ with the corresponding integral operator, we find
(check the statement in local projective coordinates)
\begin{align}\label{E-index-set}
H_0 \in \Psi^{\, 2, 0, E, E}_{\textup{phg}}(M), \quad 
E:=\left\{ \left(-\frac{n-1}{2} + \sqrt{\lambda + \frac{n-1}{2}} \, ; \, 0 \right) 
\mid \lambda \in \textup{Spec} \, \Delta_F\right\}.
\end{align}
Now by the standard procedure, cf. \cite{Mel:TAP} and \cite{MazVer}, 
$H_0$ can be refined at td, staying polyhomogeneous with 
same index sets, such that it solves the heat equation up to an error 
$$
(\partial_t + L) H_0 = \textup{Id} + R
$$
where $\beta^* R$ vanishes to infinite order at td and tf, with the index set $E$ at lf.
Now, in contrast to \cite{MazVer}, $\beta^* R$ is not polyhomogeneous unless the higher order term $h$
in the conical metric is smooth up to $x=0$. Generally we can only conclude 
like in \eqref{V}, that
$$
\beta^* R = \rho_{\textup{ff}}^{- \dim M -2 + \gamma} \, \rho_{\textup{rf}}^{-2+\gamma} \, L^\infty(\mathscr{M}^2_h),
\quad \textup{i.e.} \ R \in \Psi^{\, \gamma, \infty, E, -2+\gamma}(M).
$$
We wish to invert $(\textup{Id} + R)$ and consider the following formal Neumann series
$$
\textup{Id} + \sum_{k=1}^\infty R^k =: \textup{Id} + P,
$$ 
where the compositions $R^k$ are defined with convolution in time. Such a series is generally
referred to as a Volterra series. In fact a slightly finer analysis, cf. \cite{Mel:TAP} and also 
\cite{BGV}, shows that this formal series actually converges. While \cite{BGV} refers to the 
case of compact manifolds, and \cite{Mel:TAP} to manifolds with cylindrical ends, the estimates translate
immediately to our setting, with convergence being a general feature of such Volterra series.\medskip

We conclude in view of Theorem \ref{heat-calculus-thm} that $P \in \Psi^{\, \gamma, \infty, 0, -2+\gamma}(M)$.
At the left boundary face lf we can be more precise. Writing $N_{\textup{lf}}(P)$ 
for the restriction of $P$ to lf, we in fact have the following asymptotics
\begin{align}\label{P-lf}
\beta^* P = N_{\textup{lf}}(P) + O(\rho_{\textup{lf}}^{\overline{\gamma}}), \ \rho_{\textup{lf}} \to 0,
\end{align}
where we write
$$
\overline{\gamma} := \min \left\{\gamma,  
\mu(\lambda)= -\frac{n-1}{2} + \sqrt{\lambda + \left(\frac{n-1}{2}\right)} \mid \lambda 
\in \textup{Spec} \, \Delta_F \backslash \{0\} \right\}.
$$
We may now conclude our heat parametrix construction with the following

\begin{thm}\label{heat-asymptotics}
Consider a Riemannian manifold $(M,g)$ with isolated conical singularities in the 
sense of Definition \ref{cone-metric}. Consider any $q\in \R$. In case of $q\neq 0$
we additionally impose the admissibility assumption in Definition \ref{admissible}. Then the Schr\"odinger operator 
$L = \Delta + q \cdot \scal(g)$ admits a fundamental solution $H$ to its heat equation
$$
H:= H_0 \circ (\textup{Id} + R)^{-1} = H_0 \circ (\textup{Id} + P),
$$
such that the lift $\beta^*H$ is a conormal function on $\mathscr{M}^2_h$ with
$H \in \Psi^{\, 2, 0, 0, 0}(M)$, where writing $N_{\textup{lf}}(H)$ and $N_{\textup{rf}}(H)$ 
for the restriction of $H$ to the left and right boundary hypersurfaces, respectively,
we in fact have the following asymptotics at lf and rf 
\begin{equation}\label{lf-rf-asymptotics}
\begin{split}
&\beta^* H = N_{\textup{lf}}(H) + O(\rho_{\textup{lf}}^{\overline{\gamma}}), \ \rho_{\textup{lf}} \to 0, \\
&\beta^* H = N_{\textup{rf}}(H) + O(\rho_{\textup{rf}}^{\overline{\gamma}}), \ \rho_{\textup{lf}} \to 0,
\end{split}
\end{equation}
where
$$
\overline{\gamma} := \min \left\{\gamma,  
\mu(\lambda)= -\frac{n-1}{2} + \sqrt{\lambda + \left(\frac{n-1}{2}\right)} \mid \lambda 
\in \textup{Spec} \, \Delta_F \backslash \{0\} \right\}.
$$\end{thm}
 
\begin{proof}
By construction the fundamental solution $H$ is given by $H = H_0 + H_0 \circ P$ where
$H_0 \in \Psi^{\, 2, 0, E, E}_{\textup{phg}}(M)$ and $P \in \Psi^{\, \gamma, \infty, 0, -2+\gamma}(M)$.
Consequently, by Theorem \ref{heat-calculus-thm} we conclude $H \in \Psi^{\, 2, 0, 0, 0}(M)$.
By \eqref{P-lf} and by the explicit structure of the index set $E$ in \eqref{E-index-set}, the first
statement in \eqref{lf-rf-asymptotics} on the left face asymptotics for $H$ follows. The second
statement in \eqref{lf-rf-asymptotics} follows by symmetry of $H$. 
\end{proof}
  
\begin{remark}
If $q \cdot \scal(g) \geq 0$, then one can easily identify the fundamental solution $H$ with the 
heat operator for the Friedrichs self-adjoint extension of $L$.
\end{remark}

\section{Mapping properties of the fundamental solution} \ \medskip

In this section we study how the fundamental solution $H$, 
which we also refers to as the \emph{heat operator} acts between 
spaces of functions of prescribed asymptotics. This will use the microlocal 
heat kernel description established in Theorem \ref{heat-asymptotics}. We proceed
using the notation therein and additionally introduce function spaces that specify
asymptotics at $x=0$, which remains stable under differentiation.

\begin{defn}
Let us write $\mathscr{C}_\varepsilon(F)=(0,\varepsilon) \times F$ 
for any $\varepsilon \in (0,1)$. Consider a smooth cutoff function $\phi_\varepsilon \in C^\infty(M)$
such that $\phi_\varepsilon \restriction \mathscr{C}_{\varepsilon/2}(F) \equiv 1$ and 
$\phi_\varepsilon \restriction M \backslash \mathscr{C}_{\varepsilon}(F) \equiv 0$. 
Consider any choice of local bases
$\{X_1, \cdots, X_m\}$ of $\V_b$. Then we define 
for any $f: (0,\infty) \to \R$ and any $k\in \N$
\begin{align*}
\mathscr{O}^k_e(f(x)) := \{u: &M \to \R \mid \exists \, \varepsilon > 0, C > 0 \, \forall \, \ell \leq k, \, 
(j_1, \cdots, j_\ell) : \\
&| \left( X_{j_1} \circ \cdots \circ X_{j_k} u \right) 
\cdot \phi_\varepsilon  | \leq C \, | X_{j_1} \circ \cdots \circ X_{j_k} f|, \\ &u \cdot (1-\phi_\varepsilon) \in H^2_0(M) \}.
\end{align*}
We write $\mathscr{O}(f(x)) := \mathscr{O}^0_e(f(x))$.
\end{defn}

\begin{thm}\label{mapping-theorem}
Assume $N \leq n$. We write $H$ for the heat operator acting with convolution in time, 
and denote by $H(t)$ the heat operator acting without time convolution. $H$ solves the
inhomogeneous, while $H(t)$ the homogeneous heat equation. Then for any 
$\varepsilon > 0$ we obtain the following mapping properties.
\begin{equation*}
\begin{split}
H: \mathscr{O}(x^{-N}) \to \left\{
\begin{split} 
&\mathscr{O}^2_e(x^{-N+2}), \ \, \textup{if} \, N> 2, \\
&\mathscr{O}^2_e(\log (x)), \, \textup{if} \, N= 2, \\
&\mathscr{O}^2_e(1), \quad \quad \ \, \textup{if} \, N< 2,
\end{split}\right. \
H(t) : \mathscr{O}(x^{-N}) \to \begin{split}
&\quad t^{-\frac{N}{2}-\varepsilon} \ \mathscr{O}^2_e(1) \ \cap \\
&\left\{
\begin{split} 
&\mathscr{O}^2_e(x^{-N}), \quad \textup{if} \, N> 0, \\
&\mathscr{O}^2_e(\log (x)), \textup{if} \, N= 0, \\
&\mathscr{O}^2_e(1), \qquad \ \textup{if} \, N< 0.
\end{split}\right.
\end{split}
\end{split}
\end{equation*} 
\end{thm}

\begin{proof}
Consider $u \in \mathscr{O}(x^{-N})$. Then, writing 
$\rho_*$ for a defining function of a boundary face $*$ of the heat space
$\mathscr{M}^2_h$ we obtain in view of Theorem \ref{heat-asymptotics}, 
checking the lifts to $\mathscr{M}^2_h$ e.g. in projective coordinates 
\eqref{top-coord}, \eqref{right-coord} or \eqref{d-coord} and writing $m=\dim M$
\begin{equation}
\beta^* \left(H(t, p, \widetilde{p}) u(\widetilde{p}))\right)
= \rho_{\ff}^{-N-m} \rho_{\lf}^{-N} \rho_{\rf}^0 \, \rho_{\td}^{-m} \, \rho_{\tf}^\infty \cdot G, 
\end{equation}
where $G$ is a bounded conormal function on $\mathscr{M}^2_h$.
We now proceed with uniform estimates near the various corners of the 
front face in the heat space $\mathscr{M}^2_h$. The interior regularity of $H u$ and
$H(t) u$ is classical. Therefore, by partition of unity we may 
assume that the lifted heat kernel $\beta^* H$ is compactly supported near the 
various corners of the front face, and prove the estimates in each case separately. \medskip

\emph{Estimates near the right lower corner of the front face.} 
Assume that $\beta^* H$ is compactly supported near the right lower corner of the front face. 
Then, using the projective coordinates \eqref{right-coord} we obtain for some 
constant $C>0$ and any $K \in \R$
\begin{align*}
\left| \int_0^t \int_M H(t-\wt, p, \widetilde{p}) u(\widetilde{p}) d\wt \dv_g(\widetilde{p}) \right|
&\leq C \int_x^1 \wx^{\, -N+1} s^{\, -N+n} \, d \wx, \\
\left| \int_M H(t, p, \widetilde{p}) u(\widetilde{p}) \dv_g(\widetilde{p}) \right|
&\leq C \int_x^1 \wx^{\, -N-1} \left(\frac{t}{\ \wx^2} \right)^{-K} d \wx
\\ &= C \, t^{-K}  \int_x^1 \wx^{\, -N +2K-1} d \wx.
\end{align*}

\emph{Estimates near the left lower corner of the front face.} 
Assume that $\beta^* H$ is compactly supported near the left lower corner of the front face. 
Then, using the projective coordinates \eqref{right-coord}, where the roles of $x$ and 
$\w$ are interchanged (e.g. $s= \wx / x$), we obtain for some 
constant $C>0$ and any $K \in \R$
\begin{align*}
\left| \int_0^t \int_M H(t-\wt, p, \widetilde{p}) u(\widetilde{p}) d\wt \dv_g(\widetilde{p}) \right|
&\leq C \, x^{-N+2}\int_0^1 s^{\, -N+n} ds, \\
\left| \int_M H(t, p, \widetilde{p}) u(\widetilde{p}) \dv_g(\widetilde{p}) \right|
&\leq C \, x^{-N} \left(\frac{t}{\ x^2} \right)^{-K} \int_0^1 s^{\, -N+n} ds
\\ &\leq C \, t^{-K}  \, x^{\, -N +2K}.
\end{align*}

\emph{Estimates near the top corner of the front face.} 
Assume that $\beta^* H$ is compactly supported near the top corner of the front face. 
Then, using the projective coordinates \eqref{top-coord} we obtain for some 
constant $C>0$
\begin{align*}
\left| \int_0^t \int_M H(t-\wt, p, \widetilde{p}) u(\widetilde{p}) d\wt \dv_g(\widetilde{p}) \right|
&\leq C \int_{x}^1 \rho^{\, -N+1} d \rho \, \int_0^1 \widetilde{\xi}^{-N+n} d\widetilde{\xi}, \\
\left| \int_M H(t, p, \widetilde{p}) u(\widetilde{p}) \dv_g(\widetilde{p}) \right|
&\leq C \, \rho^{\, -N} \int_0^1 \widetilde{\xi}^{-N+n} d\widetilde{\xi}
\\ &= C \, t^{\, -\frac{N}{2}} \int_0^1 \widetilde{\xi}^{-N+n} d\widetilde{\xi}
\end{align*}

\emph{Estimates near the temporal diagonal.} 
Assume that $\beta^* H$ is compactly supported near td. 
Then, using the projective coordinates \eqref{d-coord} we obtain for some 
constant $C>0$ and any $K \in \R$
\begin{align*}
\left| \int_0^t \int_M H(t-\wt, p, \widetilde{p}) u(\widetilde{p}) d\wt \dv_g(\widetilde{p}) \right|
&\leq C x^{-N+2} \int \eta d\eta  \int G \, dS \, dZ, \\
\left| \int_M H(t, p, \widetilde{p}) u(\widetilde{p}) \dv_g(\widetilde{p}) \right|
&\leq C \, x^{-N} \int \eta^{-1} G dS dZ 
\\ &= C \, t^{-K} x^{-N+2K} \int \eta^{-1+2K} G \, dS \, dZ. 
\end{align*}

From here the statement follows by evaluating the corresponding 
integrals for the different values of $N$.

\end{proof}

\section{Improved regularity of entropy minimizers} \ \medskip

We continue in the setting of a Riemannian manifold $(M,g)$ with isolated conical singularities in the 
sense of Definition \ref{cone-metric}, with the corresponding Laplace Beltrami operator $\Delta$
and the Schr\"odinger operator $L=\Delta + q \cdot \scal(g)$. In case of $q\neq 0$
we also impose on $g$ the admissibility assumption in Definition \ref{admissible}.
We can now employ the mapping properties of the fundamental solution $H$
together with essential self-adjointness of $L$ in order to derive a better asymptotics
for the minimizers of the various entropies in \S \ref{DW-results}.

\begin{prop}\label{better-asymptotics}
Consider a Riemannian manifold $(M,g)$ with isolated conical singularities and $q\in \R$. In case of $q\neq 0$
we assume that $g$ is admissible in the sense of Definition \ref{admissible}. Assume $n=\dim F \geq 3$. 
Consider $\w \in \mathscr{O}(x^{-N}) \cap L^2(M)$ for some $N \leq n/2$, such that 
$(\Delta + q \cdot \scal(g)) \w = F(\w)$ for some functional $F(\w)$ with $|F(\w)| \leq C |\log (\w) \cdot \w|$ 
for some $C>0$. Then $\w \in \mathscr{O}(1)$ admits a partial asymptotic expansion
near the conical singularity
$$\w(x,z) = \textup{const} + O(x^{\overline{\gamma}}), \ \textup{as} \ x\to 0.$$
This expansion is preserved under differentiation by $\V_b$ of any order. Here, $\overline{\gamma}>0$ is given in terms of $\lambda_1$, the smallest nonzero eigenvalue of $\Delta_F$, and $\gamma$ in \eqref{higher-order} by
$$
\overline{\gamma} := \min \left\{\gamma,  
\mu(\lambda_1)= -\frac{n-1}{2} + \sqrt{\lambda_1 + \left(\frac{n-1}{2}\right)^2} \right\}.
$$
\end{prop}

\begin{proof}
Let us write $L:= \Delta + q \cdot \scal(g)$. Recall from Corollary \ref{minimal-domain-2} that the maximal 
domain $\dom_{\max}(L) = H^2_2(M)$ is the domain of the unique self-adjoint extension of $L$. 
By definition of the maximal domain, $\w \in \dom_{\max}(L)$, since $\w, L \w \in L^2(M)$. Consequently
$\w \in H^2_2(M)= \dom_{\max}(\Delta)$ lies in the domain of the unique self-adjoint
extension of $\Delta$. We can construct another solution to $L u = F(\w)$ using the fundamental solution $H=e^{-t\Delta}$
of $\Delta$ as follows
\begin{equation}\label{H-w-representation}
\begin{split}
\widetilde{F}(\w) &:= F(\w)-q\cdot \scal(g)\w, \\
u(t) &:= \int_0^t \int_M H(t-\wt, p, \widetilde{p}) F(\w)(\widetilde{p}) d\wt \dv(\widetilde{p})
\\ &+ \int_M H(t, p, \widetilde{p}) \w(\widetilde{p}) \dv(\widetilde{p}) 
=: (H * \widetilde{F}(\w))(t) + H(t) \w. 
\end{split}
\end{equation}
By construction $(\w - u)$ solves the initial value problem
$$
(\partial_t + \Delta) (\w - u) = 0, \ (\w - u)(0)=0.
$$
By Theorem \ref{mapping-theorem}, $u(t) \in \mathscr{O}(x^{-N} \log(x)) \subset L^2(M)$, 
uniformly in $t\geq 0$ and hence $\| u(t)\|_{L^2(M)}$ is continuous up to $t=0$, since $N \leq n/2$. 
Similar estimates as in Theorem \ref{mapping-theorem} yield
$\Delta u(t_0) \in \mathscr{O}(x^{-N}) \subset L^2(M)$ for any $t_0>0$ and hence 
$u(t_0)  \in \dom_{\max}(\Delta)$. Hence $(\w-u(t_0)) \in \dom_{\max}(L)$ 
lies in the domain of the unique self-adjoint extension of $\Delta$ and we may compute using integration by parts 
\begin{align*}
\partial_t \| \w-u(t) \|^2_{L^2(M)} = - (\Delta (\w-u(t)), (\w-u))_{L^2(M)} 
= - \| \nabla (\w-u(t))\|^2_{L^2(M)} \leq 0.
\end{align*}
Consequently, $\| \w-u(t) \|^2_{L^2(M)}$ is monotonically decreasing in $t$.
Since $u(0)=\w$ and by continuity of $\| u(t)\|_{L^2(M)}$ up to $t=0$, 
we conclude $\| \w-u(t) \|_{L^2(M)} \equiv 0$. Hence $u(t) = \w$ for any $t\geq 0$.
\medskip

Now for any fixed $t_0>0$ we can obtain an improved asymptotics for $u(t_0)$
and hence also for $\w$. Note that $\scal(g) \w \in \mathscr{O}(x^{-N-2+\gamma})$. 
For any $\varepsilon>0$ we find $F(\w) \in \mathscr{O}(x^{-N}\log(x)) \subset \mathscr{O}(x^{-N-\varepsilon})$. We write
$\gamma':= \min \{\gamma, 2+\varepsilon\}$. By Theorem \ref{mapping-theorem} we conclude that 
for each fixed $t_0>0$
\begin{equation*}
u(t_0) \in \left\{
\begin{split} 
&\mathscr{O}^2_e(x^{-N+\gamma'}), \, \, \textup{if} \, N> \gamma', \\
&\mathscr{O}^2_e(\log (x)), \ \textup{if} \, N= \gamma', \\
&\mathscr{O}^2_e(1), \quad \quad \ \, \textup{if} \, N< \gamma'.
\end{split}\right. \end{equation*} 
Argueing iteratively, we may improve the asymptotics of $\w$ step by step and 
conclude that $\w = \mathscr{O}^2_e(1)$. Due to the
expansion \eqref{lf-rf-asymptotics} we conclude that $\w$ admits a partial 
asymptotics of the form
$$\w(x,z) = \w(0,z) + O(x^{\overline{\gamma}}), \ \textup{as} \ x\to 0,$$
where $\w(0,z)$ lies in the kernel of $\Delta_F$ and hence is in fact constant in $z$.
Stability of the asymptotic expansion under differentiation by $\V_b$ follows from the 
representation \eqref{H-w-representation} and conormality of the fundamental solution kernel $H$. 
\end{proof}

We conclude the section with an obvious consequence of Proposition \ref{better-asymptotics}
and Theorems \ref{DW-1}, \ref{DW-2}, \ref{DW-3} where we now write
$\Delta \equiv \Delta_g$ in order to indicate dependence on the conical 
metric $g$.

\begin{cor}\label{asymptotics_minimizers}
Consider a Riemannian manifold $(M,g)$ with isolated conical singularities and $q\in \R$. In case of $q\neq 0$
we assume that $g$ is admissible in the sense of Definition \ref{admissible}. Assume $n=\dim F \geq 3$. 
Consider a minimizer $\w_g$ in the definition of the $\lambda$-functional, shrinker or the expander entropy. 
Then $\w_g \in \dom_{\max}(\Delta_g)$ lies in the domain of the unique self-adjoint extension of $\Delta_g$
and moreover admits for any $k \in \N$ a partial asymptotic expansion
\begin{equation*}\begin{split}
&\w_g(x,z) = \textup{const} + O(x^{\overline{\gamma}}), \\ 
&|\nabla_g^k\w_g|_g(x,z)=O(x^{\overline{\gamma}-k}). 
\end{split} \qquad  \textup{as} \ x\to 0. \end{equation*}
\end{cor}

This proves our first main result, stated in Theorem \ref{asymptotics_minimizers2}. In the next three sections, we use this result as an essential analytic tool to prove the Theorems \ref{steadysolitonricciflat2} and Theorem \ref{monotonicity} of this paper.
%

\section{Perelman's $\lambda$-functional and steady Ricci solitons} \medskip
As in the smooth setting (see e.g. \cite{CHI}) one shows that the first variation of $\lambda$\index{Lambda@$\lambda$-functional} is given by
\begin{align}\label{lambdafunctional}\lambda(g)'(h)=-\int_M\langle h,\ric (g)+\nabla^2 f_g\rangle_g e^{-f_g}\dv_g,
\end{align}
where $f_g=-2\log(\w_g)$\index{$f_g$, minimizer realizing $\lambda(g)$} is the minimizer\index{minimizer} realizing $\lambda(g)$ and $h\in C^\infty_0(M,S)$ is a symmetric $2$-tensor supported away from the singularity.
The Euler Lagrange equation for $f_g$ is
\begin{align}\label{eulerlagrangelambda}
-2\Delta_g f-|\nabla f_g|^2+\scal (g)=\lambda(g),
\end{align}
where $\Delta_g$ denotes the Laplace Beltrami operator of the metric $g$.
\begin{lem}\label{lambda-diffeo-invariance}
	Let $(M^m,g)$ be a compact manifold with an isolated conical singularity and suppose that $\scal (g_F)=(n-1)n$, where $n=m-1=\dimn(F)$. Then,
	\begin{enumerate}
		\item[(i)] if $m\geq5$, $\lambda(g)'(\mathcal{L}_Xg)=0$ for all vector fields such that $\mathcal{L}_Xg\in C^{1,\alpha}_{\textup{ie}}(M,S)_{-2}$.
		\item[(ii)] assertion (i) holds for $m=4$, provided that at least one of the following conditions is satisfied for some $\epsilon > 0$
		\begin{itemize}
			\item[a)] $\w_g(x)\to0 $ as $x\to 0$,
			\item[b)] $|\ric (g)|=O(x^{-2+\epsilon})$ as $x\to 0$ and a) does not hold,
						\item[c)] $\mathcal{L}_Xg\in C^{1,\alpha}_{\textup{ie}}(M,TM)_{-2+\epsilon}$ for some $\epsilon>0$.
			\end{itemize}
	\end{enumerate}
\end{lem}
\begin{proof}Let us start with the proof of (i).
Let $m\geq 5$, $X$ be a vector field such that $\mathcal{L}_Xg\in C^{1,\alpha}_{\textup{ie}}(M,S)_{-2}$.
We first show that there exists a $\delta>0$ such that we can choose $X\in C^{2,\alpha}_{\textup{ie}}(M,TM)_{-1-\delta}$.
 Consider the map $P:Y\mapsto \frac{1}{2}\mathcal{L}_Yg$.
Its formal adjoint $P^*$ is given by $-\mathrm{div}$ and it is easily checked that $P^*P=\Delta_g-\mathrm{Ric}_g$, where in this case, $\Delta_g$ and $\ric_g$ denote the connection Laplacian and the Ricci operator acting on $1$-forms, respectively.
  Choose $\delta\geq0$ small such that $1+\delta$ is a nonexceptional value of $P^*P$, i.e.
\begin{align}
P^*P:  C^{2,\alpha}_{\textup{ie}}(M,TM)_{-1-\delta}\to C^{1,\alpha}_{\textup{ie}}(M,TM)_{-3-\delta}
\end{align}  
 is Fredholm. Then it is standard to show that the operator
\begin{align}\label{operatorP}
P:  C^{2,\alpha}_{\textup{ie}}(M,TM)_{-1-\delta}\to C^{1,\alpha}_{\textup{ie}}(M,S)_{-2-\delta}
\end{align}
has a closed image which is given by the $L^2$-orthogonal complement of
\begin{align}\label{divker}\ker \left\{\mathrm{div}: C^{1,\alpha}_{\textup{ie}}(M,S)_{-2-\delta}\to  C^{0,\alpha}_{\textup{ie}}(M,TM)_{-3-\delta}\right\}.
\end{align}
Now we claim that
\begin{align*}\mathcal{L}_Xg\in C^{1,\alpha}_{\textup{ie}}(M,S)_{-2}\subset C^{1,\alpha}_{\textup{ie}}(M,S)_{-2-\delta}
\end{align*} is also orthogonal to \eqref{divker}
althogh we a priori do not know anything about the behaviour of $X$ at the conical singularity. We have the $L^2$-orthogonal decompositions
\begin{align*}
C^{\infty}_0(M,S)&=P(C^{\infty}_0(M,TM))\oplus \ker\mathrm{div}\cap C^{\infty}_0(M,S),\\
C^{1,\alpha}_{\textup{ie}}(M,S)_{-2-\delta}&=P(C^{2,\alpha}_{\textup{ie}}(M,TM)_{-1-\delta})\oplus\ker\mathrm{div}\cap C^{1,\alpha}_{\textup{ie}}(M,S)_{-2-\delta},
\end{align*}
where $C^{\infty}_0$ denotes sections which are supported away from the singularity. Because $C^{\infty}_0(M,S)$ and $P(C^{\infty}_0(M,TM))$ are dense in $C^{1,\alpha}_{\textup{ie}}(M,S)_{-2-\delta}$ and the closed subspace $P(C^{2,\alpha}_{\textup{ie}}(M,TM)_{-1-\delta})$, respectively, $\ker\mathrm{div}\cap C^{\infty}_0(M,S)$ is also dense in $\ker\mathrm{div}\cap C^{1,\alpha}_{\textup{ie}}(M,S)_{-2-\delta}$. 
Due to integration by parts, $\mathcal{L}_Xg$ is orthogonal to the space $ \ker{\mathrm{div}}\cap C^{\infty}_0(M,S)$. By continous embedding $C^{1,\alpha}_{\textup{ie}}(M,S)_{-2-\delta}\subset L^2$ and by denseness, $\mathcal{L}_Xg$ is also orthogonal to $\ker\mathrm{div}\cap C^{1,\alpha}_{\textup{ie}}(M,S)_{-2-\delta}$ which proves the claim. 

Therefore,  $\mathcal{L}_Xg$ lies in the image of $P$ in \eqref{operatorP} and we may assume $X\in C^{2,\alpha}_{\textup{ie}}(M,TM)_{-1-\delta}$ from now on. Now for any  symmetric $2-$tensor $h\in C^\infty(\overline{M},S)$ we can estimate
on a small $\epsilon$ neighborhood $\mathscr{C}_\varepsilon(F)$ around the singular point, for some uniform $C>0$
\begin{align*}
\int\limits_{\mathscr{C}_\varepsilon(F)}\langle h , \ric (g)+\nabla^2 f_g\rangle_g e^{-f_g}\dv_g
\leq C\cdot \sup_{\mathscr{C}_\varepsilon(F)} \, (x^{2+\delta}|h|)\cdot ( x^{2}|\ric(g)+\nabla^2f_g|)\cdot|e^{-f_g}|,
\end{align*}
which holds because $4+\delta<5\leq m$. From here we conclude 
\begin{align}\label{integral_lambda_estimate}
\lambda(g)'(h)\leq C\left\|h\right\|_{\alpha,-2-\delta}\left\| \ric(g)+\nabla^2f_g\right\|_{\alpha,2}\left\|e^{-f_g}\right\|_{\alpha,0}
\end{align} 
and the right hand side is finite by Corollary \ref{asymptotics_minimizers} and $f_g=-2\log\w_g$. 
Now take a sequence of vector fields $X_i$ with support outside of a conical singularity such that $X_i\to X$ in the space $C^{2,\alpha}_{\textup{ie}}(M,TM)_{-1-\delta}$. Due to \eqref{integral_lambda_estimate} and dominated convergence, we have
$$\lim_{i\to\infty}\lambda(g)'(\mathcal{L}_{X_i}g)=\lambda(g)'(\mathcal{L}_{X}g),$$ because $\mathcal{L}_{X_i}g\to \mathcal{L}_{X}g$ in $C^{1,\alpha}_{\textup{ie}}(M,S)_{-2}$. Due to diffeomorphism invariance, we know that $\lambda(g)'(\mathcal{L}_{X_i}g)=0$ for all $i\in\N$ and therefore, the first statement (i) follows.

Let us now prove (ii) and assume $m=4$. As in case, let $X$ be a vector field such that $\mathcal{L}_Xg\in C^{1,\alpha}_{\textup{ie}}(M,S)_{-2}$. To show the existence of $\delta>0$ such that we can choose $X\in C^{2,\alpha}_{\textup{ie}}(M,TM)_{-1-\delta}$, we have to adapt the argument from (i) slightly because $C^{1,\alpha}_{\textup{ie}}(M,S)_{-2}$ cannot be embedded to $L^2(M,S)$. In this case, it follows from $L^2$-pairing that the closed image of the operator
\begin{align}
P:  C^{2,\alpha}_{\textup{ie}}(M,TM)_{-1-\delta}\to C^{1,\alpha}_{\textup{ie}}(M,TM)_{-2-\delta}
\end{align}
is given for any $\epsilon>0$ by the $L^2$-orthogonal complement of
\begin{align}\label{divker2}\ker \left\{\mathrm{div}: C^{1,\alpha}_{\textup{ie}}(M,S)_{-2+\delta+\epsilon}\to  C^{0,\alpha}_{\textup{ie}}(M,TM)_{-3+\delta+\epsilon}\right\}.
\end{align}
 Because $\beta:=2-\delta-\epsilon$ satisfies $-\beta-2+\delta>-4$, the notion of orthogonal complement makes sense in this case. It is now shown as in (i) that $\div\ker\cap C^{\infty}_0(M,S)$ is dense in \eqref{divker2} from which we then conlude that we can choose $X$ as claimed. In case a), we have $e^{-f_g}=\w_g^2=O(x^{2\gamma'})$ by Corollary \ref{asymptotics_minimizers} and (provided that $\delta<2\gamma'$) we may replace \eqref{integral_lambda_estimate} by
\begin{align}\label{integral_lambda_estimate_a}
\lambda(g)'(h)\leq C\left\|h\right\|_{\alpha,-2-\delta}\left\| \ric(g)+\nabla^2f_g\right\|_{\alpha,-2}\left\|e^{-f_g}\right\|_{\alpha,+2\gamma'},
\end{align}
which still holds in dimension $4$. In case b), and provided that case a) does not hold, we use the inequality
\begin{align}\label{integral_lambda_estimate_b}
\lambda(g)'(h)\leq C\left\|h\right\|_{\alpha,-2-\delta}\left\| \ric(g)+\nabla^2f_g\right\|_{\alpha,-2+\epsilon}\left\|e^{-f_g}\right\|_{\alpha,0},
\end{align}
which holds as $m=4$ and $\delta<\epsilon$ and which we can use because because $\nabla^2f_g=O(x^{\gamma'-2})$ due to Corollary \ref{asymptotics_minimizers} if a) does not hold. In case c), we replace \eqref{integral_lambda_estimate} by
\begin{align}\label{integral_lambda_estimate_c}
\lambda(g)'(h)\leq C\left\|h\right\|_{\alpha,-2+\epsilon-\delta}\left\| \ric(g)+\nabla^2f_g\right\|_{\alpha,-2}\left\|e^{-f_g}\right\|_{\alpha,0},
\end{align}
which holds if $\delta<\epsilon$. Moeroever, if $1-\epsilon+\delta$ is nonexceptional, the operator
\begin{align*}
P:  C^{2,\alpha}_{\textup{ie}}(M,TM)_{-1+\epsilon-\delta}\to C^{0,\alpha}_{\textup{ie}}(M,S)_{-2+\epsilon-\delta}
\end{align*}
is Fredholm and we argue as above to conclude that we may choose $X\in C^{2,\alpha}_{\textup{ie}}(M,TM)_{-1+\epsilon-\delta} $. In all these cases, the assertion is shown as in part (i).
	\end{proof}

\begin{thm}\label{steadysolitonricciflat}
	Let $(M^m,g)$ be a compact manifold with an isolated conical singuarity and suppose that $\scal (g_F)=(n-1)n$, where $n=m-1=\dimn(F)$. Then,
\begin{itemize}	\item[(i)] if $m\geq5$, and $(M,g)$ is a steady Ricci soliton, it is Ricci flat.
	\item[(ii)] assertion (i) holds for $m=4$, provided that at least one of the following conditions is satisfied for some $\epsilon > 0$
	\begin{itemize}
		\item[a)] $\w_g(x)=O(x^{\epsilon}) $ as $x\to 0$,
		\item[b)] $|\ric (g)|=O(x^{-2+\epsilon})$ as $x\to 0$.
	\end{itemize}
\end{itemize}	
\end{thm}
\begin{remark}
In contrast to Lemma \ref{lambda-diffeo-invariance}, we did not add a case c) here, because case c) in \ref{lambda-diffeo-invariance} is equivalent to b) due to the Ricci soliton equation. Analogous assertions in the shrinker and the expanding case will be made later on but we will not add a corresponding remark there.
\end{remark}
\begin{proof}[Proof of Theorem \ref{steadysolitonricciflat}]We first show that under the conditions in $(i)$ and $(ii)$ that $(M,g)$ is gradient and that $X=\frac{1}{2}\grad f_g$. If $(i)$ holds, $\ric (g)+\mathcal{L}_Xg=0$ for some vector field $X$. Due to the Ricci soliton equation and Corollary \ref{asymptotics_minimizers}, the vector field $\frac{1}{2}\grad f_g-X$ satisfies the assumption of Lemma \ref{lambda-diffeo-invariance}. We can therefore conclude that
\begin{align*}
0=\lambda(g)'(\ric (g)+\mathcal{L}_Xg)&=\lambda(g)'(\ric (g)+\nabla^2f_g)\\&=-\int_M |\ric (g)+\nabla^2f_g|_g^2e^{-f_g}\dv_g,
\end{align*}	
which implies that $(M,g)$ is gradient. Note also that the all the terms are defined due to \eqref{integral_lambda_estimate}.
In case ii) a) or b), we use Lemma \ref{lambda-diffeo-invariance} (ii) combined with \eqref{integral_lambda_estimate_a} and \eqref{integral_lambda_estimate_b}, respectively, to show that $(M,g)$ is gradient. In both cases, we get $\ric (g)+\nabla^2f_g=0$.
Taking the trace of this equation and inserting in \eqref{eulerlagrangelambda} yields
\begin{align*}
-\Delta_g f-|\nabla f_g|^2=\lambda(g).
\end{align*}
Let $f^{||}:=\int_Mf_ge^{-f_g}\dv_g$. Because $\int_M e^{-f_g}\dv_g=1$, $f^{\perp}:=f_g-f^{||}$ satisfies $ \int_Mf^{\perp}e^{-f_g}\dv_g=0$. We obtain
\begin{align*}
0&=\lambda(g)\int_Mf^{\perp}e^{-f_g}\dv_g\\
&=-\int_M (\Delta_g f_g+|\nabla f_g|^2)f^{\perp}e^{-f_g}\dv_g\\
&=-\int_M (\Delta_g f^{\perp}+\langle \nabla f_g,\nabla f^{\perp}\rangle )f^{\perp}e^{-f_g}\dv_g\\
&=-\lim_{\epsilon\to0}\int_{M\setminus \mathscr{C}_\varepsilon(F)}
(\Delta_g f^{\perp}+\langle \nabla f_g,\nabla f^{\perp}\rangle )f^{\perp}e^{-f_g}\dv_g\\
&=-\lim_{\epsilon\to0}\int_{M\setminus \mathscr{C}_\varepsilon(F)} |\nabla f^{\perp}|^2e^{-f_g}dV_g+\lim_{\epsilon\to0}\int_{\partial \mathscr{C}_\varepsilon(F)}\nabla_{\nu}f^{\perp}\cdot f^{\perp}e^{-f_g}\dv_g\\
&=-\int_{M} |\nabla f^{\perp}|^2e^{-f_g}\dv_g,
\end{align*}
which implies that $f_g$ is constant and therefore, $\ric (g)=0$. Here, $\nu$ is the unit normal of the regular boundary of $\mathscr{C}_\varepsilon(F)$ pointing towards the conical singularity. In the last equality, we used that as $x\to 0$
\begin{align*}
\nabla_{\nu}f^{\perp}\cdot f^{\perp}e^{-f_g}=O(x^{-1}),
\end{align*}
which holds by $f_g=-2\log(\w_g)$ and Corollary \ref{asymptotics_minimizers}.
	\end{proof}
\begin{thm}\label{monotonicity_lambda}
	Let $(M^m,g)$, $m\geq 4$ be a compact manifold with an isolated conical singularity that admits an admissible metric $g_0$. Then $\lambda$ is monotonically increasing along the Ricci de Turck flow in Theorem \ref{RF}, 
	starting at $g$ and constant if and only if $(M,g)$ is Ricci flat.
\end{thm}
\begin{proof}
Let $g(t)$ be the Ricci de Turck flow starting at $g$. Then, all $g(t)$ admit the same admissible conical metric $g_0$. We have that the de Turck vector field satisfies $|\mathcal{L}_{W(t)}g(t)|=O(x^{-2})$. Let $f(t):=f_{g(t)}$ and $\w(t):=\w_{g(t)}$. For $X(t)=\grad f(t)+W(t)$, we have
\begin{align*}
\lambda(g(t))'(\mathcal{L}_{X(t)}g(t))=0.
\end{align*}
If $m\geq5$, this follows from Lemma \ref{lambda-diffeo-invariance} (i). If $m=4$, this follows from Lemma \ref{lambda-diffeo-invariance} (ii) a) or b). As a conclusion, we get
\begin{align*}
\lambda(g(t))'(\\partial_t g(t))&=\lambda(g(t))'(-2\ric (g(t))+\mathcal{L}_{W(t)}g(t))\\
&=-2\lambda(g(t))'(\ric (g(t))+\nabla^2f(t))\\
&=2\int_M |\ric (g(t))+\nabla^2f(t)|^2e^{-f(t)}\dv_{g(t)}\geq0.
\end{align*}
Note that the equalities all make sense due to \eqref{integral_lambda_estimate} if $m\geq5$ and \eqref{integral_lambda_estimate_c} and $\eqref{integral_lambda_estimate_a}$ (if $\w(t)\to0$ as $x\to0$). The second assertion is a consequence from Theorem \ref{steadysolitonricciflat} i), if $m\geq5$ and (ii) b) if $m=4$.
	\end{proof}

\section{Perelman's shrinker entropy and shrinking Ricci solitons} \medskip
Recall that the infimum in the definition of $\nu_-(g)$ exists and is realized by a pair $(f_g,\tau_g)$ if $\lambda(g)>0$.
As in the smooth setting one shows that the first variation of $\nu_-$\index{shrinker entropy} is given by 
\begin{align*}\nu_-(g)'(h)=-\frac{1}{(4\pi\tau_g)^{m/2}}\int_M\left\langle \tau_g(\ric(g)+\nabla^2 f_g)-\frac{1}{2}g,h\right\rangle e^{-f_g}\dv_g,
\end{align*}
where $(f_g,\tau_g)$ realizes $\nu_-(g)$ and $h$ is a symmetric $2$-tensor supported away from the singularity.
In this case, a pair $(f_g,\tau_g)$ realizing $\nu_-(g)$ satisfies the Euler-Lagrange equations
\begin{align}\label{nueulerlagrange}\tau_g(2\Delta_g f_g+|\nabla f_g|^2-\scal(g))-f_g+m+\nu_-(g)&=0,\\
\label{nueulerlagrange1.5}\frac{1}{(4\pi\tau_g)^{m/2}}\int_M f_g e^{-f_g}\dv_g&=\frac{m}{2}+\nu_-(g),
\end{align}
see e.g. \cite[p.\ 5]{CZ12}. 
\begin{lem}\label{numinus-diffeo-invariance}
	Let $(M^m,g)$ be a compact manifold with an isolated conical singuarity with $\lambda(g)>0$ and suppose that $\scal (g_F)=(n-1)n$, where $n=m-1=\dimn(F)$. Then,
	\begin{enumerate}
		\item[(i)] if $m\geq5$, $\nu_-(g)'(\mathcal{L}_Xg)=0$ for all vector fields $X$ such that $\mathcal{L}_Xg\in C^{1,\alpha}_{\textup{ie}}(M,S)_{-2}$.
		\item[(ii)] assertion (i) holds for $m=4$, provided that at least one of the following conditions is satisfied for some $\epsilon > 0$
		\begin{itemize}
			\item[a)] $\w_g(x)\to0 $ as $x\to 0$,
			\item[b)] $|\ric (g)|=O(x^{-2+\epsilon})$ as $x\to 0$ and a) does not hold,
			\item[c)] $\mathcal{L}_Xg\in C^{1,\alpha}_{\textup{ie}}(M,S)_{-2+\epsilon}$.
		\end{itemize}
	\item[(iii)] For any dimension $m\geq4$ and any constant $c\in\R$, we have $\nu_-(g)'(c\,g)=0$.
	\end{enumerate}
\end{lem}
\begin{proof}
	Parts (i) and (ii) are exactly shown as in Lemma \ref{lambda-diffeo-invariance}, where the corresponding assertion for the $\lambda$-functional is proven. Part (iii) follows from scale-invariance of $\nu_-$ so that $\nu_-((1+ct)g)$ is defined and constant in $t$	 and the fact that $\nu_-(g)'(c\,g)$ is a priori finite by an analogue of \eqref{integral_lambda_estimate_c} and the estimates on $f_g$ implied by Corollary \ref{asymptotics_minimizers}. 
	\end{proof}
\begin{thm}\label{shrinkingsolitongradient}
	Let $(M^m,g)$ be a compact manifold with an isolated conical singuarity and suppose that $\lambda(g)>0$ and $\scal (g_F)=(n-1)n$, where $n=m-1=\dimn(F)$. Then,
	\begin{itemize}	\item[(i)] if $m\geq5$, and $(M,g)$ is a shrinking Ricci soliton, it is gradient.
		\item[(ii)] assertion (i) holds for $m=4$, provided that at least one of the following conditions is satisfied for some $\epsilon > 0$
		\begin{itemize}
			\item[a)] $\w_g(x)\to0 $ as $x\to 0$,
			\item[b)] $|\ric (g)|=O(x^{-2+\epsilon})$ as $x\to 0$.
		\end{itemize}
	\end{itemize}	
\end{thm}
\begin{proof}If $(M,g)$ is a shrinking Ricci soliton, then $\ric (g)+\mathcal{L}_Xg=c\,g$ for some vector field $X$ and a constant $c>0$.
	We argue as in the proof of Theorem \ref{steadysolitonricciflat} and apply Lemma \ref{numinus-diffeo-invariance}
	to $Y=\frac{1}{2}\grad f_g-X$ and $(c-\frac{1}{2\tau})g$. The case by case analysis is done as in the fist 
	part of the proof of Theorem \ref{steadysolitonricciflat} and we get that $\ric (g)+\nabla^2f_g=\frac{1}{2\tau_g}g$ which implies the desired result.
	\end{proof}
\begin{thm}\label{monotonicity_nu_minus}
	Let $(M^m,g)$, $m\geq 4$ be a compact manifold with an isolated conical singularity that admits an admissible metric $g_0$. Suppose in addition that $\lambda(g)>0$. Then $\nu_-$ is monotonically increasing along the Ricci de Turck flow in Theorem \ref{RF} (with a suitable normalization, c.f. the remark below) starting at $g$ and constant if and only if $(M,g)$ is a shrinking Ricci soliton.
\end{thm}
\begin{remark}
		By a suitable normalization, we mean that the flow satisfies an evolution equation of the form
	\begin{align*}
	\partial_t g(t)=-2\ric (g(t))+\mathcal{L}_{W(t)}g(t)+\mathcal{F}(g(t))\,g(t),
	\end{align*}
	where $\mathcal{F}$ is a smooth functional on the space of metrics. Such a flow is equivalent to the standard Ricci de Turck flow by a family of rescalings.
	\end{remark}
\begin{proof}[Proof of Theorem \ref{monotonicity_nu_minus}]
The proof is completely analogous to the proof of Theorem \ref{monotonicity_lambda} and uses Lemma \ref{numinus-diffeo-invariance}.
\end{proof}
In the above theorems and in particular in the assertion of Theorem \ref{shrinkingsolitongradient}, we made the restrictive assumption that $\lambda(g)>0$ in order to make sure that the expander entropy is defined. We can overcome this problem by using a simpler variant of it by defining\footnote{The constant $\frac{1}{2}$ in the definition of 
$\mu_-$ is chosen such that it corresponds to the soliton constant $c=1$ in \eqref{def_soliton}.}
\begin{align*}
\mu_-(g)=\nu_-(g,\frac{1}{2}),
\end{align*}
which has the advantage to be defined without the restriction on $\lambda(g)$. Its first variation is
\begin{align}\mu_-(g)'(h)=-\frac{1}{2(2\pi)^{m/2}}\int_M \langle \ric(g)-g+\nabla^2 f_g,h\rangle e^{-f_g}\dv_g,
\end{align}
and the Euler-Lagrange equation of the minimizer $f_g$ is
\begin{align}\frac{1}{2}(2\Delta_g f_g+|\nabla f_g|^2-\scal)-f_g+m+\mu_-(g)=0,
\end{align}
see e.g. \cite{SW15}. The slight differences in the first variation and the Euler-Lagrange equation here and in \cite{SW15} are due to a different normalization in the definition.
Then, statements (i) and (ii) of Lemma \ref{numinus-diffeo-invariance} holds for $\mu_-$ and it can be used to prove
\begin{thm}\label{shrinkingsolitongradient2}
	Let $(M^m,g)$ be a compact manifold with an isolated conical singuarity and $\scal (g_F)=(n-1)n$, where $n=m-1=\dimn(F)$. Then,
	\begin{itemize}	\item[(i)] if $m\geq5$, and $(M,g)$ is a shrinking Ricci soliton, it is gradient.
		\item[(ii)] assertion (i) holds for $m=4$, provided that at least one of the following conditions is satisfied
		\begin{itemize}
			\item[a)] $\w_g(x)\to0 $ as $x\to 0$,
			\item[b)] $|\ric (g)|=O(x^{-2+\epsilon})$ as $x\to 0$.
		\end{itemize}
	\end{itemize}	
\end{thm}
The proof is exactly the same as in Theorem \ref{shrinkingsolitongradient}, with the only difference that we have to first rescale $g$ to ensure that the soliton constant equals $c=1$. This is because $\nu_+$ is not scale-invariant and so we can not use an analogy of Lemma \ref{numinus-diffeo-invariance} (iii) here. Furthermore, we get
\begin{thm}\label{monotonicity_nu_minus2}
	Let $(M^m,g)$, $m\geq 4$ be a compact manifold with an isolated conical singularity that admits an admissible metric $g_0$. Then $\mu_-$ is monotonically increasing along the Ricci de Turck flow 
	\begin{align*}
	\partial_tg(t)=-2\ric (g(t))+\mathcal{L}_{W(g(t))}g(t)+2g(t)
	\end{align*}
	starting at $g$ and constant if and only if $(M,g)$ is Einstein with constant $1$.
\end{thm}

\section{The expander entropy and expanding Ricci solitons} \medskip
Recall that the infimum in the definition of $\nu_+(g)$ exists and is realized by a pair $(f_g,\tau_g)$ if $\lambda(g)<0$.
As in the smooth setting one shows that the first variation of $\nu_+$\index{shrinker entropy} is given by 
\begin{align*}\nu_+(g)'(h)=-\frac{1}{(4\pi\tau)^{m/2}}\int_M\left\langle \tau_g(\ric (g) +\nabla^2 f_g)-\frac{1}{2}g,h\right\rangle e^{-f_g}\dv_g,
\end{align*}
where $(f_g,\tau_g)$ realizes $\nu_+(g)$ and $h$ is a symmetric $2$-tensor supported away from the singularity.
In this case, a pair $(f_g,\tau_g)$ realizing $\nu_+(g)$ satisfies the Euler-Lagrange equations
\begin{align}\label{nupluseulerlagrange}\tau_g(2\Delta_g f_g+|\nabla f_g|^2-\scal(g))+f_g-m+\nu_+(g)&=0,\\
\label{nupluseulerlagrange1.5}\frac{1}{(4\pi\tau_g)^{m/2}}\int_M f_g e^{-f_g}\dv_g&=\frac{m}{2}-\nu_+(g),
\end{align}
see e.g. \cite[p.\ 5]{CZ12}. 
\begin{lem}\label{nuplus-diffeo-invariance}
	Let $(M^m,g)$ be a compact manifold with an isolated conical singuarity with $\lambda(g)<0$ and suppose that $\scal (g_F)=(n-1)n$, where $n=m-1=\dimn(F)$. Then,
	\begin{enumerate}
		\item[(i)] if $m\geq5$, $\nu_+(g)'(\mathcal{L}_Xg)=0$ for all vector fields $X$ such that $\mathcal{L}_Xg\in C^{1,\alpha}_{\textup{ie}}(M,S)_{-2}$.
		\item[(ii)] assertion (i) holds for $m=4$, provided that at least one of the following conditions is satisfied
		\begin{itemize}
			\item[a)] $\w_g(x)\to0 $ as $x\to 0$,
			\item[b)] $|\ric (g)|=O(x^{-2+\epsilon})$ as $x\to 0$ and a) does not hold,
			\item[c)] $\mathcal{L}_Xg\in C^{1,\alpha}_{\textup{ie}}(M,S)_{-2+\epsilon}$.
		\end{itemize}
		\item[(iii)] For any dimension $m\geq4$ and any constant $c\in\R$, we have $\nu_-(g)'(c\,g)=0$.
	\end{enumerate}
\end{lem}
\begin{proof}This is completely analogous to the proof of Lemma \ref{numinus-diffeo-invariance}.
\end{proof}
\begin{thm}\label{expandingsolitoneinstein}
	Let $(M^m,g)$ be a compact manifold with an isolated conical singuarity and suppose that $\lambda(g)<0$ and $\scal (g_F)=(n-1)n$, where $n=m-1=\dimn(F)$. Then,
	\begin{itemize}	\item[(i)] if $m\geq5$, and $(M,g)$ is an expanding Ricci soliton, it is negative Einstein.
		\item[(ii)] assertion (i) holds for $m=4$, provided that at least one of the following conditions is satisfied
		\begin{itemize}
			\item[a)] $\w_g\to0 $ as $x\to 0$,
			\item[b)] $|\ric (g)|=O(x^{-2+\epsilon})$ as $x\to 0$.
		\end{itemize}
	\end{itemize}	
\end{thm}
\begin{proof}At first, one shows that under the assumptions of the theorem, that $(M,g)$ is a gradient soliton and that 
	\begin{align*}
	\ric (g)+\nabla^2f_g=-\frac{1}{2\tau_g}g.
	\end{align*}
	This is shown by a case-by-case analysis exactly as in the proof of Theorem \ref{shrinkingsolitongradient}. To show that $(M,g)$ is actually negative Einstein, one argues similarly as in the second part of the proof of Theorem \ref{steadysolitonricciflat}:
	At first, we see that $\scal (g)=\Delta_gf_g-\frac{m}{2\tau_g}$. Inserting in \eqref{nupluseulerlagrange} yields
	\begin{align*}
	\tau_g(\Delta_g f_g+|\nabla f_g|^2) +f_g-\frac{m}{2}+\nu_+(g)=0.
	\end{align*}
	Let $f^{||}:=(4\pi\tau_g)^{-m/2}\int_Mf_ge^{-f_g}\dv_g$. Because $\int_M e^{-f_g}\dv_g=(4\pi\tau_g)^{m/2}$, $f^{\perp}:=f_g-f^{||}$ satisfies $ \int_Mf^{\perp}e^{-f_g}\dv_g=0$. We obtain
	\begin{align*}
	0&=(\nu_+(g)-\frac{m}{2})\int_Mf^{\perp}e^{-f_g}\dv_g\\
	&=-\tau_g\int_M (\Delta_g f_g+|\nabla f_g|^2 +\frac{1}{\tau_g}f_g)f^{\perp}e^{-f_g}\dv_g\\
	&=-\tau_g\int_M (\Delta_g f^{\perp}+\langle\nabla f_g,\nabla f^{\perp}\rangle +\frac{1}{\tau_g}f^{\perp})f^{\perp}e^{-f_g}\dv_g\\
		&=-\lim_{\epsilon\to0}\tau_g\int_{M\setminus B_{\epsilon}(p)} (\Delta_g f^{\perp}+\langle\nabla f_g,\nabla f^{\perp}\rangle +\frac{1}{\tau_g}f^{\perp})f^{\perp}e^{-f_g}\dv_g\\
			&=-\lim_{\epsilon\to0}\tau_g\int_{M\setminus B_{\epsilon}(p)} (|\nabla f^{\perp}|^2+\frac{1}{\tau_g}(f^{\perp})^2 )e^{-f_g}\dv_g
			+\lim_{\epsilon\to0}\tau_g\int_{\partial B_{\epsilon}(p)}\nabla_{\nu}f^{\perp}f^{\perp}e^{-f_g}\dv_g\\
	&=-\tau_g\int_M  (|\nabla f^{\perp}|^2+\frac{1}{\tau_g}(f^{\perp})^2 )e^{-f_g}\dv_g,
	\end{align*}
	where we used Corollarary \ref{asymptotics_minimizers}.
	Therefore, $ f^{\perp}\equiv0$ which implies that $f_g$ is constant. Consequently,  $\ric (g)=-\frac{1}{2\tau_g} g$. Here, $\nu$ is the unit normal of $B_{\epsilon}(p)$ pointing towards $p$. All these equalities hold if $m\geq4$ due to the asymptotic behaviour of $f_g$.
\end{proof}

\begin{thm}\label{monotonicity_nu_plus}
	Let $(M^m,g)$, $m\geq 4$ be a compact manifold with an isolated conical singularity that admits an admissible metric $g_0$. Suppose in addition that $\lambda(g)<0$. Then $\nu_+$ is monotonically increasing along the Ricci de Turck flow (possibly with a suitable normalization) starting at $g$ and constant if and only if $(M,g)$ is negative Einstein.
\end{thm}
\begin{proof}
	The proof is completely analogous to the proof of Theorem \ref{monotonicity_lambda} and uses Lemma \ref{numinus-diffeo-invariance}.
\end{proof}
In the above theorems and in particular in the assertion of Theorem \ref{expandingsolitoneinstein}, we made the restrictive assumption that $\lambda(g)<0$ in order to make sure that the expander entropy is defined. We can overcome this problem by using a simpler variant of it by defining\footnote{The constant $\frac{1}{2}$ in the definition of 
$\mu_-$ is chosen such that it corresponds to the soliton constant $c=1$ in \eqref{def_soliton2}.}
\begin{align*}
\mu_+(g)=\nu_+(g,\frac{1}{2}),
\end{align*}
which has the advantage to be defined without the restriction on $\lambda(g)$. Its first variation is
	\begin{align}\mu_+(g)'(h)=-\frac{1}{(2\pi)^{m/2}}\int_M \langle \ric(g)+g+\nabla^2 f_g,h\rangle e^{-f_g}\dv_g,
\end{align}
and the Euler-Lagrange equation of the minimizer $f_g$ is
\begin{align}\frac{1}{2}(2\Delta_g f_g+|\nabla f_g|^2-\scal(g))+f_g-m+\mu_+(g)=0,
\end{align}
see e.g.\ \cite[p.\ 63]{Kro_Diss}. The slight differences in the first variation and the Euler-Lagrange equation here and in \cite{Kro_Diss} are due to a different normalization in the definition.
Then, statements (i) and (ii) of Lemma \ref{nuplus-diffeo-invariance} holds for $\mu_+$ and it can be used to prove
\begin{thm}\label{expandingsolitoneinstein2}
	Let $(M^m,g)$ be a compact manifold with an isolated conical singuarity and $\scal (g_F)=(n-1)n$, where $n=m-1=\dimn(F)$. Then,
	\begin{itemize}	\item[(i)] if $m\geq5$, and $(M,g)$ is an expanding Ricci soliton, it is negative Einstein.
		\item[(ii)] assertion (i) holds for $m=4$, provided that at least one of the following conditions is satisfied
		\begin{itemize}
			\item[a)] $\w_g(x)\to0 $ as $x\to 0$,
			\item[b)] $|\ric (g)|=O(x^{-2+\epsilon})$ as $x\to 0$.
		\end{itemize}
	\end{itemize}	
\end{thm}
The proof is exactly the same as in Theorem \ref{expandingsolitoneinstein}, with the only difference that we have to first rescale $g$ to ensure that the soliton constant $c=-1$. This is because $\nu_-$ is not scale-invariant and so we can not use an analogy of Lemma \ref{nuplus-diffeo-invariance} (iii) here. To show Einsteinness, one does the same proof with $\tau=\frac{1}{2}$.
 Furthermore, we get
\begin{thm}\label{monotonicity_nu_plus2}
	Let $(M^m,g)$, $m\geq 4$ be a compact manifold with an isolated conical singularity that admits an admissible metric $g_0$. Then $\mu_+$ is monotonically increasing along the Ricci de Turck flow 
	\begin{align*}
	\partial_tg(t)=-2\ric( g(t))+\mathcal{L}_{W(g(t))}g(t)-2g(t)
	\end{align*}
	 starting at $g$ and constant if and only if $(M,g)$ is Einstein with constant $-1$.
\end{thm}

\end{document}